\documentclass{amsart}

\usepackage{enumitem}
\usepackage{amssymb}
\usepackage{latexsym}
\usepackage[table]{xcolor}
\usepackage{graphicx}
\usepackage{tikz}
\usepackage{comment}
\allowdisplaybreaks

\usepackage[displaymath,mathlines,pagewise]{lineno}
\usetikzlibrary{arrows}

\usepackage{hyperref}
\usepackage[dvistyle, colorinlistoftodos]{todonotes}

\theoremstyle{plain}
\newtheorem{prop}{Property}[section]
\newtheorem{theorem}[prop]{Theorem}
\newtheorem{proposition}[prop]{Proposition}
\newtheorem{corollary}[prop]{Corollary}
\newtheorem{lemma}[prop]{Lemma}

\newtheorem{definition}[prop]{Definition}

\theoremstyle{remark}
\newtheorem{remark}[prop]{Remark}

\newtheorem{observation}[prop]{Observation}

\newcommand{\Z}{{\mathbb Z}}
\newcommand{\N}{{\mathbb N}}
\newcommand{\E}{{\mathbb E}}
\def\I{\mathcal I}
\def\A{\mathcal A}
\def\S{\mathcal S}

\def\FF{\mathcal F}

\title{Counting configuration-free sets in groups}

 \author{Juanjo Ru\'e}
\address{(JR) Freie Universit\"at Berlin, Institut f\"ur Mathematik und Informatik, Arnimallee 3, 14195 Berlin, Germany}
\email{jrue@zedat.fu-berlin.de}
\thanks{J.\,R.~was partially supported by the FP7-PEOPLE-2013-CIG project CountGraph (ref. 630749), the Spanish projects MTM2014-54745-P and MTM2014-56350-P, the DFG within the Research Training Group \emph{Methods for Discrete Structures} (ref. GRK1408), and the \emph{Berlin Mathematical School}.}

\author{Oriol Serra}
\address{(OS) Department of Mathematics, Universitat Polit\`ecnica de Catalunya and Barcelona Graduate School of Mathematics, Barcelona, Spain}
\email{\tt oriol.serra@upc.edu}
\thanks{O.\,S. was supported by the Spanish Ministerio de Econom\'ia y Competitividad under project MTM2014-54745-P }

\author{Lluis Vena}
\address{(LV) Computer Science Institute of Charles University (IUUK and ITI)
Malostransk\'e n\'amesti 25, 11800 Praha 1, Czech Republic}
\email{\tt lluis.vena@gmail.com}
\thanks{L.\,V. was supported by the Center of Excellence-Inst. for Theor. Comp. Sci., Prague, P202/12/G061, and by Project ERCCZ LL1201 \emph{CORES}}


\begin{document}

  \maketitle

\begin{abstract}
We provide new examples of the asymptotic counting for the number of subsets on groups of given size which are free of certain configurations. These examples include sets without solutions to equations in non-abelian groups, and linear configurations in abelian groups defined from group homomorphisms. The results are obtained by combining the methodology of hypergraph containers joint with arithmetic removal lemmas. As a consequence, random counterparts are presented as well.
\end{abstract}


\section{Introduction}\label{s.intro}

The study of sparse (and probabilistic) analogues of results in extremal combinatorics have become a very active area of research in extremal and random combinatorics (see e.g. the survey by Conlon~\cite{Con14}).
One starting point is \emph{Szemer\'edi Theorem}~\cite{Sze75} on the existence of arbitrarily long arithmetic progressions in sets of integers with positive upper density.
This seminal result and the tools arising in its many proofs have been enormously influential in the development of modern discrete mathematics.
Nowadays a large proportion of the research in additive combinatorics is inspired by these achievements.

Sparse analogues of Szemer\'edi Theorem started in Kohayakawa, R\"odl and {\L}uczack \cite{KoLuRo96} by studying the threshold probability for a random set of the integer interval $[1,n]$ whose subsets of given density contain asymptotically almost surely (a.a.s.) $3$--term arithmetic progressions.
The extension of the result to $k$--term arithmetic progressions was a breakthrough obtained independently, and by different methods, by Conlon and Gowers \cite{CG10} and by Schacht \cite{Sch12}.
There is still  another more recent proof based on combinatorial arguments due to Saxton, and Thomason \cite{ST15} and by Balogh, Morris, and Samotij \cite{BMS14}.
The approach in the above two papers is based on a methodology building on the structure of independent sets in hypergraphs. \emph{Hypergraphs containers} (as it is named in \cite{ST15}) provides a general framework to attack a wide variety of problems which can be encoded by uniform hypergraphs.
The philosophy behind this method is that, for a large class of uniform hypergraphs which satisfy mild conditions, one can find a small collection of sets of vertices (which are called \emph{containers}) which contain all independent sets of the given hypergraph, thus providing sensible upper bounds on the number of independent sets.

In addition to important applications in combinatorics, the two works above mentioned also contain arithmetic applications, providing in particular a new proof of the sparse Szemer\'edi Theorem. One important ingredient of these proofs, explicitly exposed in \cite{BMS14}, is the so-called \emph{Varnavides Theorem}~\cite{Var59}. This is the robust counterpart of Szemer\'edi Theorem: once a set has positive density, it does not only have one but a positive proportion of the total number of $k$--term arithmetic progressions.
This phenomenon is the number theoretical counterpart of the supersaturation phenomenon in the graph setting.

Nowadays there is a rich theory dealing with these type of results, which are rephrased under the name of \emph{Arithmetic Removal Lemmas}.
The idea behind them can be traced back to  the proof of Roth's Theorem by Ruzsa and Szemer\'edi \cite{RuzSze78} and was first formulated by Green \cite{Gre05}  for a linear equation in an abelian group by using methods of Fourier analysis. The picture was complemented independently by Shapira \cite{Sha10} and by Kr\'al', Serra and Vena \cite{KRSV12} by proving a removal lemma for linear systems in the integers. These results have been extended in several directions, including arithmetic removal lemmas for a single equation in non-abelian groups, for linear systems over finite fields and for integer linear systems over finite abelian groups (see \cite{Krsv09,Sha10,KRSV12,KrSV13}).

These extensions of Green's Arithmetic Removal Lemma provide proofs of the Szemer\'edi Theorem in general abelian groups (see also \cite{szeg10}), but cannot handle the robust versions of the multidimensional Szemer\'edi Theorem (see for instance~\cite{sol04} on Furstenberg and Katznelson work \cite{furkatz78}) or, more generally, the appearance and enumeration of finite configurations in dense subsets in abelian groups (as seen in Tao \cite[Theorem~B.1]{T12}). As a consequence, the above mentioned arithmetic removal lemmas cannot be used to show the sparse counterparts of these results (see \cite{BMS14,CG10,Sch12}).

The main contribution of this paper is to present a methodology which allows us to deal with new configurations that cannot be directly treated with the previous removal lemmas. We exemplify these configurations with the following result involving an asymptotic enumeration of rectangle-free sets in abelian groups (which can be seen as a generalization of Sidon-like sets).

\begin{theorem}[Rectangles in abelian groups, Theorem~\ref{thm:square}] Let $\{G_{i}\}_{i\geq 1}$ be a sequence of finite abelian groups, $H_i,K_i$ subgroups of $G_i$ and such that $|H_i|,|K_i|, |G_i|\to \infty$. Let
	\begin{displaymath}
	S_i=\{(x,x+a,x+b,x+a+b):\, x\in G_i, a\in H_i, b\in K_i\}
	\end{displaymath}
	be the set of configurations and let
	\begin{displaymath}
	S_i^{\scriptscriptstyle{(4)}}=\{(x_1,x_2,x_3,x_4)\in S\;: x_j\neq x_{j'},\; 1\le j<j'\le 4\},
	\end{displaymath}
	be the ones in $S_i$ with pairwise distinct entries.
	Assume that $\max\{|H_i|,|K_i|\}\leq(|S_i^{(4)}|/|G_i|)^{2/3}$.
For each $\delta>0$ with $\delta<1/40$ there exists $C=C(\delta)$ and $i_0>0$ for which the following holds: for each $i\geq i_0$ and
\begin{displaymath}
t> \frac{C}{\delta}\left(\frac{|G_i|^{4}}{|S_i^{(4)}|}\right)^{1/3},
\end{displaymath}
the number of sets free of configurations in $S_i^{(4)}$ with cardinality $t$
is bounded from above by $${ 2 \delta |G_i| \choose t}.$$
\end{theorem}

The proof of  Theorem~\ref{thm:square} contains two main ingredients.
One of them  is a removal lemma for group homomorphisms  due to Vena~\cite{Vena14} which unifies and extends previous results concerning arithmetic removal lemmas. This result can be viewed as a strong generalization of the removal lemma associated to the multidimensional Szemer\'edi Theorem for abelian groups due to Tao \cite[Theorem~B.1]{T12}, and has interest by itself. The second ingredient is the hypergraph containers method developped by Saxton, and Thomason \cite{ST15} and by Balogh, Morris, and Samotij \cite{BMS14}. Additional examples of the applications of the methodology presented here  are given in Subsection~\ref{ss.example}. These examples deal with  configurations that have not been treated before where the use of the language of homomorphisms is crucial.

In order to obtain these new applications, we have translated some of the properties of similar examples studied for instance in \cite{BMS14,ST15} into an adequate framework to treat general configurations. This language is necessary, for instance, in order to identify the threshold function when studying sparse random analogues of Szemer\'edi theorem  (see Section \ref{sec:random}). Let $S\subset G^k$ be a set of configurations in an ambient set $G$.  Given  $\delta >0$,  a subset $X\subset G$ is $(\delta, S)_k$--stable if every subset of $X$ with cardinality at least $\delta |X|$ contains a configuration  in $S$ with all elements different. The next theorem is an extension of Schacht \cite[Theorem~2.3, Theorem~2.4]{Sch12} and Conlon, Gowers \cite[Theorem~2.12]{CG10} giving the threshold for a random set to be $(\delta, S)$--stable when $S$ is defined as the kernel of a group homomorphism. One needs to be careful about subconfigurations  we want to avoid as  they may contribute in the asymptotic results. This caution explains the terminology in the statement of the next theorem, which is explained in  Section \ref{subsec:def-conf}.

\begin{theorem}[Threshold function for homomorphism systems, Theorem~\ref{t.ramdom_sparse_hom}]
	Let $k$ be a positive integer and let $\{(S_i,G_i)\}_{i\geq 1}$ be a normal sequence arising from invariant group homomorphisms,  $M_i:G_i^k\to G_i^k$, with kernel $S_i$.
Let
\begin{displaymath}
	p_{(S_i,G_i)}=\max_{\ell\in[2,k]}
	\left(\frac{\alpha^k_{\ell}(S_i,G_i)}{\alpha^{k}_{1}(S_i,G_i)} \right)^{\frac{1}{\ell-1}}.
\end{displaymath}
Then, there exist constants $c_1,c_2$, depending on $\delta$ and $k$ such that
\begin{displaymath}
	 \lim_{i\to \infty}
\mathbb{P} ([G_i]_p\,\,\mathrm{is}\,\, (\delta,S_i)_k\text{-stable})=
\left\{
\begin{array}{cc}
	1 & \text{if } p\geq c_1 p_{(S_i,G_i)}, \\
	0 & \text{if } p< c_2 p_{(S_i,G_i)},\\
\end{array}
\right.
\end{displaymath}
where $[G_i]_p$ denotes a binomial random set in $G_i$ with probability $p$.
\end{theorem}

The notion of normal sequence is fully described in Definition \ref{d.normal}.
Another example of application concerns the enumeration of  sets of given size which do not contain solutions of a given linear system. By building upon  \cite{ST15} and \cite{BMS14}, the framework presented in this paper allows us  to complete the picture of \cite[Theorem~2.10]{ST12} and obtain the following (we refer the reader to Section \ref{subsec:sys_lin_eq} for undefined terminology).

\begin{theorem}[Theorem~\ref{thm_systems}]
Let $A$ be a $k\times m$ irredundant matrix, $k>m$,  with integer entries and maximum rank.
Then, for every positive $\beta$ there exists constants $C=C(A,\beta)$ and $n_0=n_0(A,\beta)$ such that if $n\geq n_0$ and $t\geq C n^{1-1/m_A}$, then  the number of solution--free subsets of size $t$ of $[1,n]$ to the system of equations $A\textbf{x}=0$  is at most
$$
\binom{\beta n}{t}.
$$
\end{theorem}
The  paper is organized as follows.  In Section~\ref{s.prelim} we recall the terminology and the main result of the hypergrah containers method, which is one of the main ingredients of our approach,  as stated in \cite{BMS14}.
In Section \ref{subsec:def-conf} we introduce the notions of configuration systems and the supersaturation property, together with some parameters related to them. Section \ref{s.main_tool} is devoted to prove the version of the hypergraph containers method (Theorem \ref{t.count_conf_1}) that we use in our framework.
The study of random sparse versions of Szemer\'edi theorem for configurations systems  is carried out in Section \ref{sec:random}.
All the previous framework is used in Section \ref{s.families_examples} to present several applications  to configuration systems defined by homomorphisms, threshold functions for random versions, discussion of several specific examples related to cubes in finite abelian groups, specifications in integer intervals, linear equations and some examples of configurations in nonabelian groups. Finally, we briefly discuss further research in Section~\ref{s.non_ab}.


\section{Preliminaries}\label{s.prelim}

In this section we recall the main theorem from \cite{BMS14}.
Theorem~\ref{thm:cont1} is the statement which eventually leads to  counting  the number of independent sets in hypergraphs. We use it to  count solution-free sets in configuration systems as stated in Theorem~\ref{t.count_conf_1}. Particularizations of Theorem~\ref{t.count_conf_1} can be found in Section~\ref{s.families_examples}.

Let  $H=(V,E)$ be a $k$--uniform hypergraph  with $v(H)$ vertices and $e(H)$ edges. A family $\FF$ of subsets of $V(H)$ is said to be \emph{increasing} if, given $A\in \FF$ and $B\subset V(H)$ with $A\subset B$, then $B\in \FF$.
Given an increasing family $\FF$ of subsets of $V(H)$, the hypergraph $H$ is said to be \emph{$(\FF, \varepsilon)$-dense} if
$$
e(H[A])\ge \varepsilon\, e(H),\; \mbox{ for each } A\in \FF,
$$
where $H[A]$ stands for the hypergraph induced by the vertices in $A$.
The degree $d_H(T)$ of a set $T\subset V(H)$ is the number of edges of $H$  which contain $T$ and
$$
\Delta_{\ell}(H)=\max \left\{ d_H(T): T\subset {V(H)\choose \ell}\right\}.
$$
The family of independent sets of $H$ is denoted by $\I (H)$.
\begin{theorem}[Balogh, Morris, Samotij, Theorem~2.2 in \cite{BMS14}]\label{thm:cont1} For every $k \in \N$ and all positive $c$ and $\varepsilon$, there exists a positive constant
$C=C(k,\varepsilon,c)$ such that the following holds. Let $H$ be a $k$-uniform hypergraph and let $\FF \subset 2^{V (H)}$
be an increasing family of sets such that $|A| > \varepsilon v(H)$ for all $A\in \FF$. Suppose that $H$ is
$(\FF, \varepsilon)$--dense and $p \in (0, 1)$ is such that, for every $\ell \in [1,k]$,
$$
\Delta_{\ell}(H)\le cp^{\ell -1}\frac{e(H)}{v(H)}.
$$
Then there is a family $\S\in {V(H)\choose \le C pv(H)}$ and functions $f:\S\to 2^{V(H)}\setminus \FF$ and $g:\I (H)\to \S$ such that
$$
g(I)\subset I\; \mbox{ and } I\setminus g(I)\subset f(g(I)).
$$
\end{theorem}
Roughly speaking, Theorem \ref{thm:cont1} reads as follows: in a $k$-uniform hypergraph $H$ satisfying certain natural conditions,  each independent set $I$ of $H$ contains an small subset $g(I)$ (its \emph{fingerprint}) such that all sets labeled with the same fingerprint are essentially contained in a single (small) set $f(g(I))$. The notion of hypergraph containers was developed independently by Saxton and Thomason in \cite{ST12}.

\section{Systems of configurations} \label{subsec:def-conf}

This section introduces the main definitions used in this  paper. We start with the notion of system of configurations.

\begin{definition}[System of configurations]\label{d.sys_conf}
Let $k$ be a positive integer, let $G$ be a finite set and let
$S\subset G^k$. The pair $(S,G)$ is said to be a \emph{system of configurations} of degree $k$.
\end{definition}

Although not necessary, for most of the applications in Section~\ref{s.families_examples}, we ask for a group structure on $G$. Additionally, the set $S$ can be the kernel of a linear map $M$ such as in Section~\ref{s.hom} and Section~\ref{s.cube}, but this is not the case in Section~\ref{ss.nonabelian}. In this latter case, we say that $M$ induces $S$, or that the system $(S,G)$ arises from $M$.
We say that $(g_1,\ldots,g_k)\in G^k$ is a \emph{solution} of the system $(S,G)$ if $(g_1,\ldots,g_k)\in S$. $S^{(j)}$ denotes the subset of
solutions $(g_1,\ldots,g_k)\in S$ which have precisely $j$ different values.
For a given set $U=\{u_1,\ldots,u_m\}\subset [1,k]$, let $\pi_U$ denote the projection
\begin{align*}
\pi_U:G^k & \to G^{m}\\
(g_1,\ldots,g_k)&\mapsto (g_{u_1},\ldots,g_{u_m})
\end{align*}
which keeps the coordinates indexed by the elements in $U$.
For $i\in[1,k]$, let us define the \emph{$i-$th $(S,G)-$degree of freedom} as
\begin{equation*}\label{e.alpha}
\alpha_i=\max_{\substack{U\subset[1,k]\\|U|=i}} \max_{(g_1,\ldots,g_i)\in G^i} \left\{\left|S\cap \pi_U^{-1}(g_1,\ldots,g_i)\right|\right\}.
\end{equation*}
Additionally, we define the \emph{restricted $i-$th $(S,G)- $degree of freedom} as the quantity
\begin{equation*}\label{e.alpha_k}
\alpha^k_i=\max_{\substack{U\subset[1,k]\\|U|=i}} \max_{(g_1,\ldots,g_i)\in G^i} \left\{\left|S^{(k)}\cap \pi_U^{-1}(g_1,\ldots,g_i)\right|\right\}.
\end{equation*}
The $i$--th degree of freedom is an upper bound for the number of solutions which share a given $i$--th tuple.  This notion appears naturally in this context and can be found for instance in R\"odl and Ruci\'nski \cite{rodruc97}. It  plays the role of the edge density in the study of subgraphs in the random graph model of Erd\H{o}s and R\'enyi \cite{erdren60}.

The following definition is inspired by Varnavides Theorem~\cite{Var59}, which gives a robust version of Roth's Theorem~\cite{Roth52}. It describes  the supersaturation phenomenon:
\begin{definition}[Varnavides property, V-property]\label{d.varna_prop}
The system of configurations $(S,G)$ of degree $k$ is said to fulfill
the \emph{Varnavides property}, or V-property, if for every $\varepsilon>0$ there exist a
$\gamma=\gamma(\varepsilon,k)$
such
that, for any $X\subset G$ with $|X|\geq \varepsilon |G|$,
\begin{equation*}\label{eq.integer_part}
|X^k\cap S|\geq  \gamma |S|.
\end{equation*}
A sequence of systems $\{(S_i,G_i)\}_{i\geq 1}$ of degree $k$
is said to satisfy the V-property if $\gamma$ is the same function for each member of the family and only depends on $\varepsilon$.
\end{definition}


\section{Main tool and proof}\label{s.main_tool}
Theorem~\ref{t.count_conf_1} is an adaptation of \cite[Lemma~4.2]{BMS14} to count the number of solution-free sets for systems of configurations. Subscripts in constants identify the the definition or statement to which they refer to.

\begin{theorem}[Counting independent sets for configuration systems] \label{t.count_conf_1}
Let $k$ be a fixed positive integer and $\delta>0$. Let $(S,G)$ be a system of configurations of degree $k$ satisfying the
V-property with function
$\gamma=\gamma_{\text{\ref{d.varna_prop}}}$.
Write $n=|G|$. For each $i\in[1,k]$, let $\alpha^k_i$ be the restricted $i-$th $(S,G)-$degree of freedom.

Assume that each subset of $G$ with more than $\frac{\delta n}{ 2}$
elements contains a configuration in $S^{(k)}$.
Then, for each $t$ such that
\begin{displaymath}
	t\geq C \frac{|G|}{\delta}  \max_{\ell\in[2,k]}\left\{ \left(\frac{ \alpha^k_{\ell}}{\alpha^k_{1}} \frac{1}{k}{k\choose \ell} \right)^{\frac{1}{\ell-1}} \right\} \text{ and } t\leq \frac{\delta n}{ 2}
\end{displaymath}
with
\begin{displaymath}
	C=C_{\text{\ref{thm:cont1}}}\left(k,\,\frac{\xi}{|S^{(k)}|},\,\alpha_1^k \frac{(k-1)!\; |G|}{|S^{(k)}|}\right),\,\, \xi=\max\left\{(\gamma-1)|S|+|S^{(k)}|,\frac{\delta n}{2}\right\},
\end{displaymath}
there are at most
\begin{displaymath}
t\left[\frac{2e}{\delta^2}\right]^{\delta t}{\delta n\choose t}
\end{displaymath}
sets of size $t$ with no solution in $S^{(k)}$. If we assume that $\delta=\min\{\beta/2,1/40\}$, then the bound can be rewritten as
\begin{displaymath}
	\binom{\beta n}{t}.
\end{displaymath}
\end{theorem}

Theorem \ref{t.count_conf_1} deals with  sets with no solutions having its entries pairwise distinct. Some solutions may still be contained in the set but then at least two of its entries coincide. In order to apply the result to count sets free of solutions with some identical entries, one can construct a different configuration system obtained by identifying these equal entries.
Let us also recall that the constant $C$ in Theorem~\ref{t.count_conf_1} depends on $\gamma$ through the $(\mathcal{F},\varepsilon)$-dense condition in Theorem~\ref{thm:cont1}.

\begin{proof}
	The proof follows the lines of the arguments of \cite[Lemma~4.2]{BMS14}.
	We include the details for completeness.
We consider the $k$-uniform hypergraph $H$ whose vertex set is $V(H)=G$.
Observe that each solution $\textbf{x}\in S^{(k)}$ (all the variables having different values) of the system $(S,G)$ defines a set of size $k$ in $G$ (namely, forgetting the order of the variables).
We define the edge set as
$$	E(H)=\left\{ \{x_1,\ldots,x_k\} : (x_{\sigma(1)},\ldots,x_{\sigma(k)})\in S^{(k)} \text{ for a permutation $\sigma$ on $k$ elements}\right\} $$
by using the identification between vertices of the hypergraph and elements of the group.
Thus $v(H)=n$ and $e(H)$ satisfies
\begin{equation}\label{e.eq1}
\frac{|S^{(k)}|}{k!}\leq e(H)\leq|S^{(k)}|.
\end{equation}
Observe that every independent set of $H$ defines a solution--free set of the configuration system $(S,G)$ restricted to the set of solutions $S^{(k)}$.

Consider now the family of sets with more than $\delta n$ vertices: $\FF =\{F\subseteq V(H): |F|\geq \delta n\}.$
We shall show that $H$ satisfies the conditions of Theorem \ref{thm:cont1} with respect to the family $\FF$.
The family $\FF$ is clearly increasing. Since $(S,G)$ satisfies the V-property, given any set $F\in\FF$, there are more than $\gamma|S|$ solutions involving elements in $F$. In particular,  there are at least
\begin{equation*}\label{e.e1}
	\gamma|S|-(|S|-|S^{(k)}|)=(\gamma-1)|S|+|S^{(k)}|
\end{equation*}
solutions whose entries are pairwise distinct  and belong to  $F$. On the other hand,  since each set of size $\frac{\delta n}{2}$ contains a solution in $S^{(k)}$, then any set of size at least $\delta n$ contains at least $\frac{\delta n}{2}$ solutions in $S^{(k)}$. Hence, there are at least $\xi=\max\left\{(\gamma-1)|S|+|S^{(k)}|,\frac{\delta n}{2} \right\}$ solutions in $S^{(k)}\cap F^k$.

Therefore there are at least $\xi/k!$ edges in each set $F$.
The total number of edges is $e(H)$, which satisfies the relations in \eqref{e.eq1}.
Let $\varepsilon=\xi/e(H).$
Then $H$ is $(\FF,\varepsilon)$-dense with an $\varepsilon$ such that
\begin{displaymath}
\frac{\xi}{|S^{(k)}|}\leq \varepsilon\leq \frac{\xi}{|S^{(k)}|} k!.
\end{displaymath}
Let us now check the conditions concerning the degrees. For each $\ell \in [1,k]$ we have that
\begin{displaymath}
	\Delta_{\ell}(H)\leq\alpha^k_{\ell} {k \choose \ell}
\end{displaymath}
as this is the maximum number of solutions in $S^{(k)}$ containing a given  subset of $\ell$ vertices. Choose $c$ with
\begin{equation*}\label{e.def_c}
	c= k \alpha^k_1 \frac{v(H)}{e(H)}\geq k \alpha^k_1 \frac{|G|}{|S^{(k)}|}.
\end{equation*}
Then
\begin{displaymath}
	\Delta_1(H)\leq c \frac{e(H)}{v(H)}.
\end{displaymath}
The parameter $p$ in Theorem~\ref{thm:cont1} is chosen as
\begin{equation*}\label{e.def_p}
	p=
	\max_{\ell\in[2,k]} \left\{ \left(\frac{1}{c}\frac{v(H)}{e(H)} \alpha^k_{\ell} {k\choose \ell} \right)^{\frac{1}{\ell-1}}\right\}.
\end{equation*}
Then we have
\begin{displaymath}
	\Delta_{\ell}(H)\leq c p^{\ell-1} \frac{e(H)}{v(H)}.
\end{displaymath}
Therefore, we can apply Theorem~\ref{thm:cont1} and obtain a constant $C=C_{\text{\ref{thm:cont1}}}(k, \varepsilon, c)$, a set $\S\subset {V(H)\choose  \leq Cpn}$ and functions $f:\S\to 2^{V(H)}\setminus \FF$ and $g:\I (H)\to \S$ such that
$$
g(I)\subset I\; \mbox{ and } I\setminus g(I)\subset f(g(I)).
$$
Write $C'=C/\delta$. Let $\I (H,t)$ denote the sets in $\I(H)$ with cardinality $t$.
Since $g(I)\subset I$ and $f(g(I))\supset I\setminus g(I)$ we have
$$
|\I (H,t)|=\sum_{S\in \S} |\{I\in \I (H,t): g(I)=S\}|\le \sum_{S\in \S} {|f(S)|\choose t-g(I)}.
$$
Consider $t$ a parameter to be
\begin{displaymath}
	t\geq \frac{C}{\delta}pn=\frac{C}{\delta}p |G|.
\end{displaymath}
As $f(S) \in 2^{V(H)}\setminus \FF$, then $|f(S)|\le \delta n$ for each $S\in \S$.
Moreover, sets in $\S$ have cardinality at most  $Cpn$. Then, by the range of $t$, we have that $Cpn\le \delta t$. Therefore,
\begin{displaymath}
 \sum_{S\in \S} {|f(S)|\choose t-g(I)}\le
  \sum_{r\le \delta t} {n\choose r}{\delta n \choose t-r}\le
\sum_{r\le \delta t} \left(\frac{en}{r}\right)^r\left(\frac{t}{\delta n-t}\right)^r{\delta n\choose t}\le
  \sum_{r\le \delta t}\left(\frac{2et}{\delta r}\right)^r{\delta n\choose t}.
\end{displaymath}
The first inequality follows by considering all the possible sets of size $r=|g(I)|$, while the second one follows from ${\delta n \choose t-r}\leq  \left(\frac{t}{\delta n -t}\right)^r\frac{\delta n!}{t! (\delta n -t)!}$ and the use of Stirling approximation for ${n\choose r}$.
The third, and last, inequalities follows as there are no independent sets of cardinality larger than $\delta n/2$, so that $t\leq \delta n/2$ and therefore, $\frac{n}{\delta n -t}\leq \frac{n}{\delta n - \delta n/2}=\frac{2}{\delta}$. Recall that only solutions in $S^{(k)}$ are considered.
Observe that the range for $t$ might be empty, in which case  the result  says nothing but remains true.
Therefore,
$$|\I (H,t)|\le t\left(\frac{2e}{\delta^2}\right)^{\delta t}{\delta n\choose t}.$$
The second estimate arises observing that $\binom{\delta n}{t} \leq 2^{-t} \binom{2 \delta n}{t}$ (using the bound $\binom{b}{c}\leq (b/a)^c \binom{a}{c}$ (for $a\geq b\geq c \geq 0$), $2^{1/(3 \delta)}> 2e/\delta^2$ (if $\delta \leq 1/40$) and $2^{2t/3}>t$ for $t$ nonnegative integer.
Thus if we let $\delta=\min\{\beta/2,1/40\}$ we obtain that
\begin{displaymath}
		|\I (H,t)|\le t\left(\frac{2e}{\delta^2}\right)^{\delta t}{\delta n\choose t}\leq \binom{\beta n}{t}.
\end{displaymath}
We finally observe that $C_{\text{\ref{thm:cont1}}}(k,\varepsilon,c)$ is increasing whith $c$, and increasing when $\varepsilon$ is decreasing. Indeed, it is shown in the proof of \cite[Theorem 2.2]{BMS14} that
$$C_{\text{\ref{thm:cont1}}}(k,\varepsilon,c)=(k-1)\left(\frac{1}{\delta}\log\left(\frac{1}{\varepsilon}\right)+1\right),$$
where $\delta=(ck 2^{k+1})^{-k}$ (see the proof of \cite[Proposition 3.1]{BMS14}), from which the previous claim follows, and the definition of $C$ in the statement is justified.
Since any independent set of $H$ corresponds to a solution free set with respect to $S^{(k)}$, the result follows.\end{proof}

Most of the examples we discuss in Section~\ref{s.families_examples} share some common additional features that are used to draw conclusions from Theorem~\ref{t.count_conf_1}. The following definition gathers these properties.

\begin{definition}\label{d.normal}
	A sequence $\{(S_i,G_i)\}_{i\ge 1}$  of configuration systems of degree $k$  is said to be \emph{normal} with function $\gamma$ if
\begin{enumerate}[label=C{\arabic*}]
	\item \label{cond-1} The $G_i$'s are finite and growing in size with $i$.
	\item \label{nc1} ${\displaystyle \lim_{i\to \infty} \frac{|G_i|}{|S_i|}=0}$ and $\displaystyle \lim_{i\to \infty} \frac{|G_i|^k}{|S_i|}=\infty$.
	\item \label{cond.4} $\displaystyle \lim_{i\to \infty} \frac{|S_i^{(k)}|}{|S_i|}=1$.
	\item \label{cond.3} Each  $(S_i,G_i)$ satisfies the $V$--property with a function $\gamma=\gamma_{\text{\ref{d.varna_prop}}}(\delta )$ universal for all the systems in the sequence.
\end{enumerate}
\end{definition}

The above conditions reflect the asymptotic nature of the results.
Condition \ref{nc1} states that  the number of configurations asymptotically exceeds the trivial ones for invariant systems (namely, the ones with constant entries), and that these configurations impose non--trivial restrictions, hence its number  is asymptotically smaller than the whole set of possible configurations. All the invariant systems of configurations arising from integer matrices with two more columns than rows satisfy condition \ref{nc1}. Although not crucial for many applications, these conditions are specially needed when dealing with the random sparse analogues  we treat in  Section~\ref{sec:random}. Condition \ref{cond.4} ensures that most of the solutions have pairwise distinct entries. This is again a common feature in most applications and one can reduce to configurations systems satisfying it by identifying some entries.
Additionally, in most of our applications, $G_i$ has a group structure and $S_i$ is induced by a group homomorphism. The latter constraint can be relaxed in the non--abelian setting.


\section{Random sparse results} \label{sec:random}

We address in this section the study of random sparse models.
Our objective is to extend known  sparse analogues of extremal results in additive combinatorics.
The study of these type of problems has a long history. Kohayakawa,  \L uczak and R\"odl. \cite{KoLuRo96} address  the following question (which can be interpreted as an sparse random analogue of Roth's Theorem).
Let $p$ be a probability function that may depend on $n$.
Denote by  $[n]_p$ the binomial random subset of $[1,n]$ obtained by choosing independently each element with probability $p$.
A subset $X \subseteq [1,n]$ is  \emph{$\delta-$Roth} if every subset $X' \subset X$ with  $|X'| \geq \delta |X|$  contains an arithmetic progression of length $3$.
The main result of \cite{KoLuRo96} shows that, for all $\delta>0$, the $\delta-$Roth property has a threshold.
More precisely, there exist constants $c(\delta)$ and $C(\delta)$ such that, when choosing $p \leq c(\delta) n^{-1/2}$, with probability tending to $0$ the random set $[n]_p$  is not  $\delta - $Roth while, if $p \geq C(\delta) n^{-1/2}$, then with probability tending to $1$ the random set $[n]_p$ is $\delta-$Roth.

The extension of this property to arithmetic progressions of length $r$ (with threshold function equal to $n^{-1/(r-1)}$) was  studied in  Conlon and Gowers~\cite{CG10} and independently in Schacht~\cite{Sch12}. Later on, another proof of this extension was obtained from the hypergraph containers method in  Balogh, Morris and Samotij~\cite{BMS14} and independently  Saxton and Thomason \cite{ST12}.

In this section we address generalizations of the above result  to systems of configurations. The main properties that we study are the following, which extend the notion of $\delta-$Roth defined before:
\begin{definition}\label{def: d-property}
Let $(S,G)$ be a system of configurations of degree $k$.
We say that a set $X \subseteq G$ is \emph{the $(\delta,S)-$stable} if, for  every subset $X' \subset X$ with  $|X'| \geq \delta |X|$, we have  $X'^k \cap S \neq \emptyset$.
We also say that $X$  is $(\delta,S)_k$--stable if in addition $X'^k \cap S^{(k)} \neq \emptyset$.
\end{definition}

We study the previous properties in the binomial random model: fix a probability $p$ (that may depend on $|G|$), and consider the binomial random set $[G]_p$ built by choosing independently each element of $G$ with probability $p$.
We observe that both the $(\delta,S)$-stable and  the $(\delta,S)_k$-stable properties are \emph{not} monotone increasing (as they are not closed by supersets). However, we will show that there is a threshold phenomenon for these properties given that the configuration system has some uniformity properties.

We prove the $0-$statement and the $1-$statement in two separate subsections. In general, there is a gap between the $0-$statement and the $1-$statement. In Subsection~\ref{s.annoying_example}, we provide conditions for the configuration systems to avoid such a gap. These conditions of uniformity will be satisfied in the examples of Section~\ref{s.families_examples}.

\subsection{$1-$statement}\label{ssec.1-statement}

In the following theorem we cover the previous results regarding the $1-$statement in an unified way in the language of systems of configurations:

\begin{theorem}\label{thm: random2}
	Let $\delta>0$ and let $\{(S_i,G_i)\}_{i\geq 1}$ be a normal sequence  of system of configurations of degree $k$ with function $\gamma$.
Write $n_i=|G_i|$ and assume that every set with more than $\delta n_i/2$ elements has a solution in $S_i^{(k)}$ and that $(\gamma-1)|S_i|+|S_i^{(k)}|>0$ for all $i$.
Let
	$$p_{(S_i,G_i)}=\max_{\ell\in[2,k]}
	\left(\frac{\alpha^k_{\ell}(S_i,G_i)}{\alpha^{k}_{1}(S_i,G_i)} \right)^{\frac{1}{\ell-1}}.$$
Then there exists $C=C(\delta, \gamma, k)>0$ such that
\begin{equation}\label{eq:prob}
\lim_{i\to \infty} \mathbb{P}([G_i]_p\,\,\mathrm{is}\,\,(\delta,S_i)-\mathrm{stable})=           1,\,\text{if } p \geq C p_{(S_i,G_i)}.
\end{equation}
\end{theorem}

\begin{proof}
The proof follows the lines of Corollary 4.1 in~\cite{BMS14}.
We write $(S,G)$ as a generic configuration system $(S_i,G_i)$ of degree $k$ and $|G|=n$.
Assume that $p \geq C p_{(S,G)}$ with $C=C_{\text{\ref{t.count_conf_1}}} \max_{\ell\in [2,k]} \left(\frac{1}{k}{k\choose \ell}\right)^{\frac{1}{\ell-1}}$.
Write $t= \frac{\delta}{2} pn$ and $\beta=\delta/2 \cdot e^{-1/\delta-1}$.
Let $X_t$ be the random variable counting the number of subsets of $[G]_p$ of size $t$ without configurations.
Let $\mathcal{E}$ be the event that $[G]_p$ does not have the $(\delta,S) -$stability.
Hence:
\begin{align} \label{eq:bound}
\mathbb{P}(\mathcal{E})&\leq \mathbb{P}\left(\mathcal{E} \cap\left[ |[G]_p| \geq \frac{1}{2}pn\right]\right)+ \mathbb{P}\left(\mathcal{E} \cap \left[|[G]_p| < \frac{1}{2}pn\right]\right)\\ \nonumber
&< \mathbb{P}\left(\mathcal{E} \cap \left[|[G]_p| \geq \frac{1}{2}pn\right]\right)+ \mathbb{P}\left( |[G]_p| < \frac{1}{2}pn\right)\\ \nonumber
& < \mathbb{P}(X_t>0)+ e^{-pn/8},
\end{align}
where we have used Chernoff's inequality (see for instance Appendix A of \cite{AlSp92}) to exponentially bound $\mathbb{P}( |[G]_p| < \frac{1}{2}pn)$.
Let us finally bound the probability $\mathbb{P}(X_t >0)$. We are under the assumptions of Theorem \ref{t.count_conf_1}, hence, for $n$ large enough,
$$\mathbb{P}(X_m >0) \leq \mathbb{E}[X_t] \leq \binom{\beta n}{t} p^t \leq \left(\frac{\beta e p n  }{t}\right)^t = \left(\frac{2 \beta e  }{\delta}\right)^t= e^{-t/\delta}.$$
Putting this bound together with the upper bound obtained in \eqref{eq:bound} gives that for $p \geq p_{(S,G)}$ we have that
$$\mathbb{P}([G]_p\,\,\mathrm{has}\,\,\mathrm{the}\,\,(\delta,S)\mathrm{-stable}) \geq 1- e^{-t/\delta}-e^{-pn/8}$$
as we wanted to show.
\end{proof}

Observe that we have shown something even stronger, as the use of Theorem~\ref{t.count_conf_1} guarantees that the solutions found will have its entries pairwise distinct. Hence,   $(\delta,S)-$stability can be replaced  by  $(\delta,S)_k-$ stability in the statement of Theorem~\ref{thm: random2}.

\subsection{$0-$statement}

In this section we are interested in solutions with pairwise distinct entries, that is,  we only consider $S^{(k)}$ as the solution set.
The strategy to prove the $0-$statement is based on an application of the Alteration Method (see e.g. \cite[Chapter~3]{AlSp92}).
We use the fact that the random variable $|[G]_p|$ is asymptotically concentrated around its expected value $p|G|$.
However, we additionally need to assume some concentration around the expected values of the different projections of the solution set:
\begin{definition}[Concentration for configuration systems]\label{d.concentration_intersection}
Let $\mathcal{G}=\{(S_i,G_i)\}_{i\geq 1}$ be a sequence of configuration systems with $n_i=|G_i|$, and let $0\leq p\leq 1$ (which may depend on $n_i$).
We say that  $\mathcal{G}$ is \emph{concentrated} for $p$ if,
for every $\varepsilon,\varepsilon'>0$, there exist an $n_0=n_0(\varepsilon,\varepsilon',k)$ such that
$$
	\mathbb{P}\left(\left| \left|\pi_U\left(S_i^{(k)}\right)\cap [G_i]_p^{|U|}\right|  - p^{|U|}\left|\pi_U\left(S_i^{(k)}\right)\right| \right|\geq \varepsilon p^{|U|}|\pi_U\left(S_i^{(k)}\right)|\right)\leq \varepsilon',
$$
for each $n_i\geq n_0$ and each $U\subset [1,k]$ with $|U|\geq 2$.
\end{definition}

In other words, $|\pi_U\left(S_i^{(k)}\right)\cap [G_i]_p^{|U|}|$ is asymptotically concentrated around its expectation.
Let us observe that a general sequence of systems of configurations may fail to satisfy such concentration. For instance,  a system of configurations in which one variable only takes a single value in $G_i$.
However, for sequences ${\mathcal  G}$ which are concentrated in the sense of  Definition \ref{d.concentration_intersection}, we can obtain a $0-$statement for a wide range of values for $p$.
The proof is divided into two parts.
As it will be shown, the V-property is not required both in Proposition~\ref{t.zero_statement0} and in Theorem~\ref{t.zero_statement}.

We start proving the $0$-statement for small values of $p$. In this case only the second part of condition \ref{nc1} in the definition of a normal sequence is needed.
\begin{proposition}[$0-$statement for configurations, small $p$]\label{t.zero_statement0}
Let $\delta>0$ and let $\{(S_i,G_i)\}_{i\geq 1}$ be a sequence of systems of configurations of degree $k$. Let $\beta_i=\frac{|G_i|^k}{|S_i^{(k)}|}$ and assume that  $\displaystyle \lim_{i\to \infty} \beta_i=\infty$. If
$$
	p \leq  f(\beta_i,|G_i|) |S_i^{(k)}|^{-1/k}
$$
for some $f$ with $f(\beta_i,|G_i|) \to_{|G_i|\to \infty} \infty$, then
$$
	 \lim_{i\to \infty}
\mathbb{P}([G_i]_p\,\,\mathrm{is}\,\, (\delta,S_i)_k\mathrm{-stable})= 0.
$$
\end{proposition}
\begin{proof}
Let $Z=|[G_i]_p|$ and  $Y=|[G_i]_p^k\cap S_i^{(k)}|$. We have $\E(Z)=p|G_i|$ and $\E(Y)=p^k|S_i^{(k)}|$.
We distinguish two cases depending on the choice of $p$.

Assume first that $p \leq h |G_i|^{-1}$, so that $\E(Z)\le h$.
Let $b=\log(\beta_i)$, choose $h=\sqrt{b}$, and write
\begin{equation*} \label{eq.prob1}
	\mathbb{P}(Y\geq 1)= \mathbb{P}([Y\geq 1]\wedge |[G_i]_p|> b) + \mathbb{P}([Y\geq 1]\wedge |[G_i]_p|\leq b).
\end{equation*}
By Markov inequality
	\begin{equation*} \label{eq.prob2}
\mathbb{P}([Y\geq 1]\wedge |[G_i]_p|> b)\leq	\mathbb{P}\left(|[G_i]_p|>b\right)\leq \frac{\E(Z)}{b}\leq \frac{h}{b}
\to 0\; (i\to\infty).
\end{equation*}
To bound the second term observe that
\begin{align} 
	\mathbb{P}&([Y\geq 1]\wedge |[G_i]_p|\leq b)=\sum_{j=k}^b \mathbb{P}([Y\geq 1]\wedge |[G_i]_p|= j)\nonumber \\
 &\leq \sum_{j=k}^b \sum_{X\in {G_i\choose j} }  \sum_{x\in X^k} \mathbb{P}(x\in S_i^{(k)}\cap [G_i]_p^k\big| [G_i]_p=X) \mathbb{P}([G_i]_p=X) \nonumber \\
&= \sum_{j=k}^b \sum_{X\in {G_i\choose j} }  \sum_{x\in X^k} \mathbb{P}(x\in S_i^{(k)}\cap [G_i]_p^k\big| [G_i]_p=X) p^j (1-p)^{|G_i|-j}\nonumber \\
&= \sum_{j=k}^b \sum_{X\in {G_i\choose j} }  \sum_{x\in X^k}
\left\{ \begin{array}{cc}
1 & \text{ if } x\in S_i^{(k)} \\
0 & \text{ if } x\notin S_i^{(k)}\\
\end{array}\right\} p^j (1-p)^{|G_i|-j} \nonumber \\
&= \sum_{j=k}^b \sum_{x\in S_i^{(k)}} |\{X\ni \{x\} | X\in {G_i \choose j}\}| p^j (1-p)^{|G_i|-j}  \nonumber \\
&=\sum_{j=k}^b |S_i^{(k)}| {|G_i|-k \choose j-k} p^j (1-p)^{|G_i|-j} \nonumber \\
&\leq b! \sum_{j=k}^b \frac{|S_i^{(k)}|}{|G_i|^k} |G_i|^j p^j (1-p)^{|G_i|-j} \nonumber \\
&\stackrel{h>1}{\leq} (b+1)! \frac{|S_i^{(k)}|}{|G_i|^k} h^b \label{eq.ei} \nonumber \\
&\leq \frac{1}{\beta_i} b^{3b/2} = \frac{3\log^2(\beta_i)}{2\beta_i} \to 0 \; (i\to \infty) \nonumber
\end{align}
Therefore, for this range of $p$ the random set $[G_i]_p$ has no solutions with high probability.

Assume now that $p \geq h |G_i|^{-1}$ and $p \leq f |S_i^{(k)}|^{-1/k}$.
In this range we will show that $[G_i]_p$ contains few solutions and there is a large subset which contains none, so that the random set is not $(\delta,S_i)_k$--stable.
By Markov inequality,
\begin{equation}\label{eq.prob4}
	\mathbb{P}(Y>a)\leq \frac{\mathbb{E}(Y)}{a}\leq \frac{f^k}{a}
\end{equation}
and
\begin{equation} \label{eq.prob3}
	\mathbb{P}(|[G_i]_p|< d)\geq 1-\frac{h}{d}.
\end{equation}
Choose $d=h^2$. As $h\to_{i\to\infty}\infty$, then \eqref{eq.prob3} tends to $1$ as $i$ increases.
Hence, writing $a=\frac{(1-\delta) d}{2}$ and
\[
f=\sqrt[2k]{a}=\sqrt[2k]{\frac{(1-\delta)d}{2}}=\sqrt[2k]{\frac{(1-\delta)h^2}{2}}=\sqrt[2k]{\frac{(1-\delta)b}{2}}=\sqrt[2k]{\frac{(1-\delta)\log(\beta_i)}{2}},
\] we get that \eqref{eq.prob4}, the value we are interested in, tends to $0$ as $i$ increases, $f$ goes to infinity with $i$, and we can delete all the solutions, with probability tending to $1$, be removing at most $(1-\delta) d$ elements (as the number of solutions is, with asymptotically high probability, smaller than $(1-\delta)d/2$, so we can delete one element in $[G_i]_p$ per each solution.)
\end{proof}

The next theorem studies the regime when $p$ is large:

\begin{theorem}[$0-$statement for configurations, large $p$]\label{t.zero_statement}
Let $\delta>0$ and let $\{(S_i,G_i)\}_{i\geq 1}$ be a family of systems of configurations of degree $k$.
Write
\begin{equation}\label{eq.0-state}
	p_{(S_i,G_i)}=\min\left\{\max_{\substack{U\subseteq [1,k]\\ |U|\geq 2}}
	\left(\frac{|G_i|
	}{|\pi_U(S_i^{(k)})|}\right)^{\frac{1}{|U|-1}},1\right\}.
\end{equation}
Assume that $\{(S_i,G_i)\}_{i\in I}$ satisfies Definition~\ref{d.concentration_intersection} with $p$ such that
\begin{equation}\label{eq.proba_1}
	\max\left\{|G_i|^{-1},|S_i^{(k)}|^{-1/k}\right\}=o(p) \text{ and }
p	\leq c p_{(S_i,G_i)}
\end{equation}
for some constant $c:=c(\delta,k,\{(S_i,G_i)\}_{i\geq 1})$.
Then, if  $p$ satisfies (\ref{eq.proba_1}),
$$
	 \lim_{i\to \infty}
\mathbb{P}([G_i]_p\,\,\mathrm{is}\,\, (\delta,S)_k\mathrm{-stable})= 0.
$$
\end{theorem}

\begin{proof}
Consider a generic system of configurations of degree $k$ in the family $(S,G)$.
Recall that $\mathbb{E}(|[G]_p|)=p|G|$.
Also, under the conditions of Definition~\ref{d.concentration_intersection}, for $U\subset [1,k]$, $|U|\geq 2$ we have that $\mathbb{E}\left(|\,[G]_p^{|U|}\cap \pi_U(S^{(k)})\,|\right)=p^{|U|} |\pi_U(S^{(k)})|$.
Consider now
\begin{equation*}
	p':=p'(U)=\left(\frac{|G|}{|\pi_U(S^{(k)})|}\right)^{\frac{1}{|U|-1}} \,\mathrm{and}\,\,\,\, p''=p'\frac{1-\delta}{4}.
\end{equation*}
Assume that $p'\leq 1$ for each choice of $U$.
Observe that
\begin{displaymath}
	p'|G|=p'^{|U|}|\pi_U(S^{(k)})|.
\end{displaymath}
This equality tells us that the expected number of elements in the random set $[G]_{p'}$ equals the expected number of solutions. Observe also that
\begin{equation*}
\left(\frac{1-\delta}{4}\right)^{|U|-1} p''|G|=(p'')^{|U|} |\pi_U(S^{(k)})|.
\end{equation*}
Let us analyze the random set $[G]_{p''}$.
With probability tending to $1$ as $|G|\to \infty$, we have that by Definition~\ref{d.concentration_intersection}
\begin{displaymath}
|\pi_U(S^{(k)})\cap [G]_{p''}^{|U|}|\leq \sqrt{2}(p'')^{|U|} |\pi_U(S^{(k)})|.
\end{displaymath}
Observe also that asymptotically almost surely
\begin{displaymath}
	|[G]_{p''}|\geq \frac{1}{\sqrt{2}} p''|G|.
\end{displaymath}
If there exist a set of relative density $\geq \delta$ in $[G]_{p''}$ with no solutions, then the set $[G]_{p''}$ will not satisfy the $(\delta,S)_k$-stable property.
On one hand,
there are, asymptotically almost surely, at most
\begin{displaymath}
	\sqrt{2}(p'')^{|U|} |\pi_U(S^{(k)})|
\end{displaymath}
configurations in $\pi_U(S^{(k)})\cap [G]_{p''}^{|U|}$.
Additionally,
\begin{displaymath}
	\sqrt{2}(p'')^{|U|} |\pi_U(S^{(k)})|=\left(\frac{1}{4}\right)^{|U|-2}\frac{\sqrt{2}}{4} (1-\delta)^{|U|-1} p'' |G|\leq \frac{1}{\sqrt{2}} \frac{1-\delta}{2} p'' |G|.
\end{displaymath}
Therefore, we can apply the Alteration Method in the following way: by avoiding a set of size $\frac{1}{\sqrt{2}} \frac{1-\delta}{2} p'' |G|$ from $[G]_{p''}$,
we can find a subset with no configurations in $\pi_U(S^{(k)})\cap [G]_{p''}^{|U|}$.
As the remaining part of $[G]_{p''}$ after removing at most $\frac{1}{\sqrt{2}} \frac{1-\delta}{2} p'' |G|$ has relative size larger than
$\delta$, we can find sets of relative density larger than $\delta$ (asymptotically almost surely).
Therefore, with probability tending to $0$, the property is fulfilled.

Since this can be done for each choice of $U$, we can take the maximum of all these probabilities $p'$ to find the maximum probability for which the $(\delta,S)_k$-stability is fulfilled with probability $0$ for some $U$ (and, hence, for the system as a whole).
Thus by picking the maximum of them, namely
\begin{equation*}
\max_{\substack{U\subseteq [1,k]\\ |U|\geq 2}} p'(U)=\max_{\substack{U\subseteq [1,k]\\ |U|\geq 2}}
	\left(\frac{|G|}{|\pi_U(S^{(k)})|}\right)^{\frac{1}{|U|-1}}
\end{equation*}
and $c=\frac{1-\delta}{4}$ the result holds.

Finally assume that $p_{(S,G)}=1$. Then, by equation \eqref{eq.0-state}, $|\pi_{V}(S^{(k)})|\leq |G|$ for some $V\subset [1,k]$, \, $|V| \geq 2$.
In such case, picking $p_{(S,G)}=1$ and $c=(1-\delta)/4$ as before also makes the result valid as, with probability tending to $1$, we will pick a set of relative density $\geq \delta$ that avoids $\pi_{V}(S^{(k)})$ completely.
Indeed, each $x\in\pi_{U'}(S^{(k)})$ considers $|V|\geq 2$ elements in $G$, each of its components.
Hence, there exist a set of size at most $|G|/2$ such that, if $[G]_{p_{(S,G)}}$ avoids it, $[G]_{p_{(S,G)}}^{|V|}\cap \pi_{V}(S^{(k)})=\emptyset$.
By the choice of $c$, this will be possible asymptotically almost surely.
\end{proof}

\subsection{Uniformity} \label{s.annoying_example}

The order of magnitude on the probability in Theorem~\ref{thm: random2} and in Theorem~\ref{t.zero_statement} do no match for a general configuration system, as the example in Appendix~\ref{s.app_uni} shows.
The following natural notion ensures an equality between the probabilities coming from Theorem~\ref{thm: random2} and Proposition~\ref{t.zero_statement}, as Proposition~\ref{p.zero_statement_uniform} shows. In particular, the example provided in Appendix~\ref{s.app_uni} does not satisfy the following definition.

\begin{definition}[$\rho-k-$uniformity]\label{d.g-uniform}
	Given $\rho>0$,
	the system $(S,G)$ is said to be \emph{$\rho-k-$uniform} (or $\rho-$uniform if $k$ is clear from the context) if, for each $U\subset [1,k]$ with $|U|\geq 2$, and for each $(x_1,\ldots,x_{|U|})\in \pi_U(S^{(k)})$, then
	$$
		\left|\pi^{-1}_U((x_1,\ldots,x_{|U|}))\cap S^{(k)}\right|\geq \rho \max_{(y_1,\ldots,y_{|U|})\in \pi_U(S^{(k)})}\left\{\left|\pi^{-1}_U((y_1,\ldots,y_{|U|}))\cap S^{(k)}\right|\right\}.
	$$
\end{definition}
In other words, the number of solutions projected to an element is, up to a constant factor, the same as its maximum number.
\begin{proposition}\label{p.zero_statement_uniform}
	Let  $\{(S_i,G_i)\}_{i\geq 1}$ be a sequence of systems of configurations.
	Write
$$
	p_{(S_i,G_i)}= 	\max_{\ell\in[2,k]} \left(\frac{\alpha^k_{\ell}}{\alpha^k_1} \right)^{\frac{1}{\ell-1}}.
$$
	Assume that $\{(S_i,G_i)\}_{i\geq 1}$
	\begin{itemize}
\item satisfies Definition~\ref{d.concentration_intersection} for $p$ with
	\begin{equation*}\label{eq.proba_4}
	\max\left\{|G_i|^{-1},|S_i^{(k)}|^{-1/k}\right\}=o(p) \text{ and }
p	\leq c p_{(S_i,G_i)}
\end{equation*}	
\item is $\rho-k-$uniform, with the same $\gamma>0$ for each of the systems.
\end{itemize}
Then there exists constants $0<c<C$, $c$ and $C$ depending on $\delta$ and on $k$, such that
\begin{equation*}
	 \lim_{i\to \infty} \mathbb{P}([G_i]_p\,\,\mathrm{is}\,\,(\delta,S_i)_k-\mathrm{stable})=\left\{ \begin{array}{c}
                                                                                                                     0,\,\,\text{if\,\,  $p \leq c p_{(S_i,G_i)}$,}  \\
                                                                                                                     1,\,\,\text{if\,\,  $p \geq C p_{(S_i,G_i)}$.}
                                                                                                                   \end{array}\right.
\end{equation*}
\end{proposition}

\begin{proof}
	Pick a generic system $(S,G)$.
Assume $\pi_{\{i\}}(S^{(k)})=G$ for all $i\in [1,k]$. Then
\begin{equation}\label{eq.adjust}
\frac{|S^{(k)}|}{|G|}=\alpha_1^k.
\end{equation}

In any system we have, by the definition of $\alpha_\ell^k$,
\begin{displaymath}
	\alpha_\ell^k\geq \max_{\substack{U\subseteq [1,k]\\ |U|=\ell}}\frac{\left|S^{(k)}\right|}{\left|\pi_U(S^{(k)})\right|}.
\end{displaymath}
If the system is $\rho-$uniform then by Definition~\ref{d.g-uniform}
\begin{displaymath}
	\gamma\alpha_{\ell}^k\leq \max_{\substack{U\subseteq [1,k]\\ |U|=\ell}}\frac{\left|S^{(k)}\right|}{\left|\pi_U(S^{(k)})\right|}.
\end{displaymath}

Now we put everything together and substitute these expressions in Theorem~\ref{t.zero_statement}. The result follows as the constant depend on $\rho$ and the precise value of the index $\ell$ that gives the maximum in $p_{(S,G)}$, hence the absolute $C$ can be obtained as a function of $k$ and $\rho$.

If $\pi_{\{i\}}(S^{(k)})\subsetneq G$, then if $|\pi_{\{i\}}(S^{(k)})|\geq |G|/2$, the proof goes through by adjusting the $C$ by a constant in \eqref{eq.adjust}. Otherwise, the 0-statement is also fulfilled by the argument at the end of Theorem~\ref{t.zero_statement} and by observing that $p_{(S_i,G_i)}\leq 1$ by the definitions of $\alpha_\ell^k$. The other cases for $p$ are covered by Proposition~\ref{t.zero_statement0}.
\end{proof}

As we shall see, the systems of configuration discussed in Section~\ref{s.families_examples} satisfy Definition~\ref{d.concentration_intersection} and Definition~\ref{d.g-uniform}.
However, our results do not clarify if the system has a gap between the $0$-statement and the $1$-statement (aside from the system is $\rho-k-$uniform).
For a further discussion on how some conditions might be relaxed, the reader is referred to the approach taken by \cite{Sch12} and the discussion in \cite[page~692]{BMS14} to adapt the approach taken here to such case.


\section{Families of examples}\label{s.families_examples}

We present in this part several applications of our framework.
In order to apply the results obtained in the previous sections, we shall see that our system of configurations satisfies the V-property (Definition~\ref{d.varna_prop}), the concentration property (Definition~\ref{d.concentration_intersection}), and the uniformity property (Definition~\ref{d.g-uniform}).

We start this section presenting configurations arising from homomorphisms between abelian groups.
In particular, we present some new examples in Subsection \ref{ss.example}.
In Subsection \ref{s.cube} we particularize the previous homomorphism setting to cubes.
Finally we discuss the prominent case of system of equations over abelian groups and equations over non-abelian groups (Subsection \ref{subsec:sys_lin_eq} and Subsection \ref{ss.nonabelian} respectively).
%
%
\subsection{Homomorphisms of finite abelian groups}\label{s.hom}
In this section we study sets free of \emph{configurations arising from homomorphisms} between finite abelian groups. That is it to say, $G$ is a finite abelian group, $M$ is a group homomorphisms $M:G^k\to G^m$ and the solution set is $S=M^{-1}(0)$ or, in general, $M^{-1}(g)$, with $g\in G^m$.
Then the system $(S,G)$ is a system of configurations of degree $k$.
A group homomorphisms $M$ is said to be \emph{invariant} if, for every $x\in G$, then $(x,\ldots,x)\in M^{-1}(0)$.
In all this section we restrict ourselves to system of configurations arising from invariant homomorphisms as they are a quite general class of systems of configurations that satisfy the V-property as Lemma~\ref{l.V-property_invariant_hom} shows.
\subsubsection{Shape of invariant systems of homomorphisms}\label{ss.shape}
As a matter of illustration of the configurations that can be encoded by the group homomorphism setting, we show in this section the canonical form of invariant homomorphisms in the integers. This will be later exploited in Subsection~\ref{s.cube} to deal with configuration in cubes.

For a given homomorphism $M:[\Z^m]^{k_1}\to[\Z^m]^{k_2}$, we can consider it as a
linear system defined by $m k_1 \times m k_2$ integer values (see \cite{Vena14}).
This allows us to define the volume of $M$ as done in~\cite{RuZU14}. Let $e_i$ denote the vector with $1$ in the $i$-th coordinate and $0$ in the rest.
Since $(e_i,\ldots,e_i)\in S\cap [1,i]^m$ for each $i\in[1,m]$, then we observe that, in each of the $mk_2$ equations, the variables corresponding with some $i$-th coordinate should sum to $0$ for each coordinate and each equation.

\begin{proposition}\label{p.solution_structure} The solutions of each invariant system $M:[\Z^m]^{k}\to [\Z^m]^{k}$ (namely, elements $(x_1,\ldots,x_k)\in S=M^{-1}(0)$) can be codified by equation of the form
\begin{displaymath}
	\left(\begin{array}{c}
		x_1 \\
		x_2 \\
		\vdots \\
		x_k
	\end{array}\right)=
	\left(\begin{array}{c}
		x_0 \\
		x_0 \\
		\vdots \\
		x_0
	\end{array}\right)+
	\lambda_1
	\left(\begin{array}{c}
		F_{1,1} \\
		F_{1,2} \\
		\vdots \\
		F_{1,k}
	\end{array}\right)+\cdots+
	\lambda_q
	\left(\begin{array}{c}
		F_{q,1} \\
		F_{q,2} \\
		\vdots \\
		F_{q,k}
	\end{array}\right)
\end{displaymath}
for fixed $q$, $F_{i,j}\in \Z^m$ (depending on $M$) and some $\lambda_i\in \Z$ and $x_0\in \Z^m$.
\end{proposition}
\begin{proof}
	One direction is clear (the right to left), as those configurations are linear and invariant. On the other direction, we proceed as in the case of linear systems of equations to find a basis. Consider the solution set and relate two solutions if their difference is an element as
$(x_0,\ldots,x_0)$. There is an extra solution (more than just the class of the zero). Take the representative with minimal $l_1$ norm and name it $F_1$. Now consider the modulus of the solution set by the possible sum of
$(x_0,\ldots,x_0)+\lambda_1 F_1$ for each $\lambda_1\in \Z$. Do the same as there are more than just one class (the class of zero). The process should end as the maximum number of degrees of freedom is $mk_1$ (both these numbers depends on $M$).
\end{proof}
Let us show an illustrative example: with the invariant homomorphisms $A:[\Z^2]^3\to[\Z^2]^2$ given by 
\begin{equation*}\label{eq.examp_hom_sys}
	A=\left(
	\begin{array}{ccc}
		\begin{pmatrix}
			1 & 0 \\
			0 & 1
\end{pmatrix} &
\begin{pmatrix}
	0&0 \\
	0 & -1 \\
\end{pmatrix}
& \begin{pmatrix}
 1& 0 \\
0 &0 \\
\end{pmatrix} \\
\begin{pmatrix}
	-1 & 1 \\
	0 & 0 \\
\end{pmatrix}&
\begin{pmatrix}
	1 & 0 \\
	0 & 0 \\
\end{pmatrix}&
\begin{pmatrix}
	0 & -1 \\
	0 & 0 \\
\end{pmatrix}
	\end{array}
	\right),
\end{equation*}
we can codify the $2$-dimensional simplices in $\Z^2$: sets of $3$ points of the type $((x,y), (x+a,y),(x,y+a))$ for $x,y,a\in \Z$.
If we are interested in configurations with $a\geq 0$, we can consider a configuration symmetric with respect to the $(x,y)$ such as: $((x,y), (x+a,y),(x,y+a),(x-a,y),(x,y-a))$.
These configuration system cannot be codified using an integer matrix with three columns (one per each point).

\subsubsection{Results for invariant homomorphism systems}

We start proving the V-property for invariant homomorphism systems, Lemma~\ref{l.V-property_invariant_hom}, using the following arithmetic removal lemma from \cite{Vena15}:
\begin{theorem}[Removal lemma for homomorphisms, Theorem~2 in \cite{Vena15}] \label{t.rem_lem_ab_gr}
	 Let $G$ be a finite abelian group and let $k$ be a positive integer. Let $M:G^k\to G^k$ be a group homomorphism.
	Let $b\in G^k$. Let $X_i\subset G$ for $i=[1,k]$, and $X=X_1\times \cdots \times X_m$ and let $S=M^{-1}(b)$.
	
	For every $\varepsilon>0$ there exists a $\delta=\delta(\varepsilon,k)>0$ such that, if
 $$\left|S\cap X\right|<\delta \left|S\right|,$$ then there are sets $X_i'\subset \left[X_i\cap \pi_{\{i\}}(S)\right]$ with $|X_i'|<\varepsilon |\pi_{\{i\}}(S)|$ and $$S\cap \left( X\setminus X'\right)=\emptyset, \text{ where } X\setminus X'=(X_1\setminus X'_1)\times \cdots \times (X_m\setminus X'_m).$$
\end{theorem}

As a consequence of this result we have the following corollary:

\begin{lemma}[V-property for invariant homomorphism systems]\label{l.V-property_invariant_hom}
	Let $(S,G)$ be a system of configurations arising from an invariant homomorphism $M$ of degree $k$.
Then, for every $\varepsilon>0$ there exist $\gamma=\gamma(\varepsilon,k)>0$ such that, for any set $A\subset G$
	with $|A|>\varepsilon |G|$, then $|A^k \cap S|\geq \gamma |S|$.
\end{lemma}
\begin{proof}
We use Theorem~\ref{t.rem_lem_ab_gr}
and proceed by contradiction.
To destroy all the configurations we should remove, at least $\varepsilon/k$ elements in each of the copies of $A$ that form the cartesian product $S^k$.
This is true because, for each $s\in A$, $(s,\ldots,s)\in S$.
Therefore, there are more than $\delta_{\text{\ref{t.rem_lem_ab_gr}}}(\varepsilon/k,k)|S|$ solutions with all its elements in $S^m$, hence proving the result with $\gamma(\varepsilon,k)=\delta_{\text{\ref{t.rem_lem_ab_gr}}}(\varepsilon/k,k)$.
\end{proof}

\begin{corollary}\label{o.counting_sol_free_hom}
	Let $k$ a fixed integer and a $\delta>0$. Then for any normal sequence $\{(S_i,G_i)\}_{i\geq 1}$ of configuration systems (coming from invariant homomorphisms), there exist $n_0$ depending on $\delta$ such that if  $|G_i|>n_0$
then for each $t$ such that
$$
	t\geq C \frac{|G_i|}{\delta}  \max_{\ell\in[2,k]}\left\{ \left(\frac{ \alpha^k_{\ell}}{\alpha^k_{1}} \frac{1}{k}{k\choose \ell} \right)^{\frac{1}{\ell-1}} \right\} \text{ and } t\leq \frac{\delta |G_i|}{ 2}
$$
there are
\begin{displaymath}
t\left[\frac{2e}{\delta^2}\right]^{\delta t}{\delta |G_i|\choose t}
\end{displaymath}
sets of size $t$ with no solution in $S_i^{(k)}$.
\end{corollary}

\begin{proof}
We use Theorem~\ref{t.count_conf_1}, so all the conditions should be verified. Observe that, by Lemma~\ref{l.V-property_invariant_hom} the system satisfies the V-property with a uniform function depending only on $k$, $\gamma=\gamma_{\text{ \ref{d.varna_prop}}}(\delta,k)$.
The condition $\xi=(\gamma-1)|S|+|S^{(k)}|>0$ and the fact that any set of size $\delta |G|/2$ has a solution with all the variables being different is satisfied for an $n_0$ sufficiently large depending on $k$ and $\delta$ (because of the $\gamma$), and on how fast \ref{cond.4} approaches $1$.
$C$ is a constant depending on $k$ and on $n_0$. Indeed, normality
implies that $C$ depends on how fast the ratio between the different solutions found in any set of relative size $\delta$ and the whole solution set approaches $1$.
\end{proof}

Corollary~\ref{o.counting_sol_free_hom}, as application of Theorem~\ref{t.count_conf_1}, is rather general and technical but its main purpose is to make the applications easy to show. Now we show the concentration property for homomorphism systems.
\begin{proposition}[Concentration for homomorphism systems] \label{p.hom_concentration}
	Let $\{S_i,G_i\}_{i\geq 1}$ be an normal family of systems arising from invariant homomorphisms.
	Let $p=p(i)\in[0,1]$ be such that
\begin{enumerate}[label=(\roman*)]
	\item $\lim_{\substack{i\to \infty}} p(i) |G_i|=\infty$.
	\item\label{cond.inside_prop_2} $\lim_{\substack{i\to \infty}} p(i) |\pi_U(S_i)|^{1/|U|}=\infty$ for each $U\subset [1,k]$, $|U|\geq 2$.
\end{enumerate}
Then the sequence satisfies Definition~\ref{d.concentration_intersection} for such $p$.
\end{proposition}
Let us remark that condition \ref{cond.inside_prop_2} implies $\lim_{\substack{i\to \infty}} p(i) |S_i^{(k)}|^{1/k}=\infty$.
%
%
%
\begin{proof}
Our strategy is to use the Second Moment Method in the form of \cite[Corollary~4.3.4]{AlSp92}.
In most of the proof, we use $G$ instead of $G_i$ to simplify the notation and the little-$o$ that appear are asymptotic considerations when $i\to\infty$.
Let $X$ denote the random variable that counts the number of solutions in $S^{(k)}\cap [G]_p^k$.
We write $X$ as $X=X_1+\cdots+X_r$, $r=|S^{(k)}|$, where $X_i$ is an indicator random variable that gets the value $1$ if the $i$-th solution in $S^{(k)}$ belongs to $[G]_p^k$, and $0$ otherwise.
Denote by $A_i$ the event associated to $X_i$.
We also write $i \sim j$ if $A_i$ and $A_j$ are not independent.
As it is proven in \cite[Corollary~4.3.4]{AlSp92}, it is enough to show that $\Delta=\sum_{i\sim j} \mathbb{P}(A_i \cap A_j)=o(\mathbb{E}(X^2))$ (for a certain range of $p$) in order to conclude that $X$ is concentrated around its average value.

We compute $\Delta$ grouping the pairs of events $(A_i,A_j)$ according to the size of the non-empty intersection.
Write $\Delta=\sum_{\ell=1}^k \Delta_\ell$, where $\Delta_i$ contains the summands $i\sim j$ in $\Delta$ such that $|A_i\cap A_j|=\ell$.
Observe that this non-empty intersection can occur in several ways ($A_i$ is a set, while $X_j$ is an ordered tuple), hence a constant correcting factor (bounded above and below by functions of $k$) has to be considered when computing $\Delta$.

We start bounding $\Delta_1$.
Observe that $\Delta_1\leq kp^{2k-1}|S^{(k)}|\frac{|S^{(k)}|}{|G|}$ up to a constant depending on $k$.
Indeed, in the bound for $\Delta_1$, the factor $k$ stands for the position of the common variable, $p^{2k-1}$ is the probability of having precisely these $p^{2k-1}$ different elements in $|A_i\cup A_j|$, the factor $|S^{(k)}|$ stands as we are checking for all the solutions $A_i$ (and then we are seeing how many solutions $A_j$ are there satisfying the first condition).
Finally, the factor $\frac{|S^{(k)}|}{|G|}$ is an upper bound on the number of solutions $A_j$ that share $1$ element with $A_i$ in the $s$-th position.
Observe that the factor $\frac{|S^{(k)}|}{|G|}$ arises as the system is invariant, hence the maps $+g:S^{(k)}\to S^{(k)}$ such that $x\to x+(g,\ldots,g)$ for each $g\in G$ creates a partition in $S^{(k)}$ according to the value on any (fixed) coordinate of the solution set (even if $S^{(k)}$ is empty).
As some of the $\leq k! \frac{|S^{(k)}|}{|G|}$ solutions $x_j$ that have $x_{j,\ell}\in A_j$ for the $\ell$ coordinate (fixed at the beginning) may also have $x_{j,\ell'}\in A_j$ for some $\ell'\neq \ell$, the bound presented is an upper bound (up to a constant depending on $k$).
We argue now on $\Delta_2$.
The arguments for $\Delta_l$, $l\geq 3$ are similar.
For $\Delta_2$, there are ${k\choose 2}$ possible choices for the coordinates to be shared. Also, there are $\pi_{U}(S^{(k)})$, with $U=\{i,j\}$, $|U|=2$ possible pairs of values $x_{i}=g_1$ and $x_{j}=g_2$. Observe that each of the $S^{(k)}$ solutions have, at most, $\frac{|S|}{|\pi_{U}(S)|}$ elements sharing the same pair of $(x_i=g_1,x_j=g_2)$. Indeed,
the number of elements in $S$ that have $(x_i,x_j)$ as the $i$-th and the $j$-th variables respectively is, either $0$, or the size of the homogeneous system with the addition of the equations $x_i=0$ and $x_j=0$ (each solution is the sum of one in the new homogeneous system plus a particular solution with the addition of $x_i=g_1$, $x_j=g_2$). Hence, $|S|/|\pi_{\{i,j\}}(S)|$ is the appropriate number when it is different from $0$.
Since $\frac{S}{\pi_{U}(S)}\leq \frac{S}{\pi_{U}(S^{(k)})}$ for each $U \in {[k] \choose 2}$,
the factor in $\Delta_2$ is, at most, a constant depending on $k$ times
\begin{displaymath}
	p^{2k-2}
	\sum_{U\in {[k]\choose 2}}\frac{|S|}{|\pi_{U}(S^{(k)})|} |S^{(k)}|.
\end{displaymath}
By normality, $\frac{|S|}{|\pi_{U}(S^{(k)})|} $ is asymptotically equivalent to $\frac{|S^{(k)}|}{|\pi_{U}(S^{(k)})|} $.
Summing all bounds for $\Delta_l$, $1\leq l \leq k$ we can bound $\Delta$ by
\begin{displaymath}
	\Delta\leq c(k) \sum_{s=1}^k  p^{2k-s}
	\sum_{U\in {[k]\choose s}}\frac{|S|}{|\pi_{U}(S^{(k)})|} |S^{(k)}|,
\end{displaymath}
for a certain constant $c(k)$ only depending on $k$. If
\begin{equation}\label{eq.alguna}
	p|G|^{2k-s}\frac{|S|}{|\pi_{U}(S^{(k)})|} |S^{(k)}|=o(p|G|^{2k} |S^{(k)}|^2)=o(\mathbb{E}(X)^2),
\end{equation}
then we can use \cite[Corollary~4.3.4]{AlSp92} to show the result. Expression (\ref{eq.alguna}) follows from
$\lim_{i\to \infty}\frac{|S_i|}{|S_i^{(k)}|}=1$, and that, by hypothesis,
$o(p(i))=\frac{1}{|G_i|}$ (so that $\mathbb{E}(X)\to \infty$), and $o(p(i))= |\pi_U(S_i^{(k)}|^{-1/|U|}$ for each $U\subset [1,k]$. Hence the result is proven.
\end{proof}
Let us observe that the condition \ref{cond.3} from normality has not been needed for Proposition~\ref{p.hom_concentration}. Additionally, the invariant systems that satisfy \ref{cond.4} also satisfy \ref{cond-1}.
\begin{proposition}\label{r.unif_hom_sys}
	Any sequence of invariant homomorphisms systems $\{(S_i,G_i)\}_{i\geq 1}$ of degree $k$ such that $\lim_{i\to \infty}\frac{|S_i^{(k)}|}{|S_i|}=1$ (Condition \ref{cond.4})
	satisfies the uniformity condition from Definition~\ref{d.g-uniform}.
\end{proposition}

The proof of Proposition~\ref{r.unif_hom_sys} is contained in the argument to prove Proposition~\ref{p.hom_concentration} when we see the bound on the number of solutions that each equation has.

%
After all the preliminary results have been shown, we are ready to prove the main result of this section.
\begin{theorem}[Threshold function for homomorphism systems]\label{t.ramdom_sparse_hom}
Let $1>\delta>0$
be a positive real number, and let $\{(S_i,G_i)\}_{i\geq 1}$ be an normal sequence of systems of degree $k$ arising from homomorphisms of finite abelian groups.
Let
\begin{equation}\label{eq.0-state_2}
	p_{(S_i,G_i)}=\max_{\ell\in[2,k]}
	\left(\frac{\alpha^k_{\ell}(S_i,G_i)}{\alpha^{k}_{1}(S_i,G_i)} \right)^{\frac{1}{\ell-1}}.
\end{equation}
Then, there exist constants $c_1,c_2$, depending on $\delta$ and $k$ such that
\begin{displaymath}
	 \lim_{i\to \infty}
\mathbb{P}([G_i]_p\,\,\mathrm{is}\,\, (\delta,S_i)_k\mathrm{-stable})=
\left\{
\begin{array}{cc}
	1 & \text{if } p\geq c_1 p_{(S_i,G_i)}, \\
	0 & \text{if } p< c_2 p_{(S_i,G_i)}.\\
\end{array}
\right.
\end{displaymath}
\end{theorem}

Let us comment that the normal condition (more precisely, the first part of \ref{nc1}, together with \ref{cond.4}) warranties that the solution set in the invariant systems has asymptotically more solutions than the trivial ones. Additionally, condition~\ref{cond.4} allows the use of Theorem~\ref{thm: random2}.

\begin{proof}
We put all the pieces gathered in Section~\ref{sec:random} and from the beginning of the current section. We use $(S,G)$ to denote a generic $(S_i,G_i)$ in the proof. The little-$o$'s are referred to asymptotic behaviour when $i\to \infty$ (or $|G|\to\infty$).

First, Theorem~\ref{thm: random2} can be applied to find the constant $c_1$ and show the first part of the result as $\lim_{i\to \infty} |G_i|=\infty$ and the system satisfy the V-property thanks to Lemma~\ref{l.V-property_invariant_hom}. Indeed, Lemma~\ref{l.V-property_invariant_hom} warranties that $\gamma$ only depends on $k$ and on $\delta$, but is independent on the particular homomorphisms or group, hence it is uniform for the family of homomorphisms. The first part of \ref{nc1}, together with \ref{cond.4} allows us to warranty that, for $|G|$ large enough depending on $\delta$ and $k$ (and on how fast the limit goes to $\infty$), any set of density $\delta |G|/2$ has solutions in $S^{(k)}$ and  condition $\xi_{\text{\ref{t.count_conf_1}}}/|S^{(k)}|>a>0$ is satisfied for some $a$ depending on $\delta$ and on the family of systems.

Let us now show the 0-statement.
In this case, there is a discrepancy between the range of probabilities given by
Proposition~\ref{p.hom_concentration} and those used in the hypothesis of Proposition~\ref{p.zero_statement_uniform} (or Theorem~\ref{t.zero_statement}) that we address in the following argument.
The uniformity of the systems of configurations coming from homomorphisms, Definition~\ref{d.g-uniform}, is satisfied by Proposition~\ref{r.unif_hom_sys}.
Hence, up to a constant factor depending on $k$ and on the family of systems, we can use \eqref{eq.0-state} as an alternative definition of $p_{(S,G)}$. Let $c_2=c_{\text{\ref{t.zero_statement}}}$ be the constant arising in the upper bound on the probability coming from Theorem~\ref{t.zero_statement}, which contains the constant $c_2$ according to the chosen definition of $p_{(S,G)}$.
Observe that due to the invariance of the system, the argument from the proof of Proposition~\ref{p.hom_concentration} regarding $\Delta_1$, together with a combination of the first part of \ref{nc1} and \ref{cond.4} implies that
\begin{displaymath}
	\max_{\substack{U\subseteq [1,k]\\ |U|\geq 2}}
	\left(\frac{|G_i|
	}{|\pi_U(S^{(k)})|}\right)^{\frac{1}{|U|-1}}\leq 1,
\end{displaymath}
hence
\begin{equation}\label{e.first_ineq_hom_0}
	p\leq c_2 p_{(S,G)}=c_2 \min\left\{\max_{\substack{U\subseteq [1,k]\\ |U|\geq 2}}
	\left(\frac{|G_i|
	}{|\pi_U(S^{(k)})|}\right)^{\frac{1}{|U|-1}},1\right\}=c_2 \max_{\substack{U\subseteq [1,k]\\ |U|\geq 2}}
	\left(\frac{|G|
	}{|\pi_U(S^{(k)})|}\right)^{\frac{1}{|U|-1}}.
\end{equation}
Since $|\pi_U(S^{(k)})|\leq |G|^{|U|}$ and $|\pi_{[1,k]}(S^{(k)})|=o(|G|^k)$, then
\begin{equation}\label{eq.hola_gramola_jajaja}
o\left(\max_{\substack{U\subseteq [1,k]\\ |U|\geq 2}} |\pi_U(S^{(k)})|^{-\frac{1}{|U|}}\right)=|G|^{-1} \text{ and }
o\left(\max_{\substack{U\subseteq [1,k]\\ |U|\geq 2}}
	\left(\frac{|G|
	}{|\pi_U(S^{(k)})|}\right)^{\frac{1}{|U|-1}}\right)=|G|^{-1}\; .
\end{equation}
Observe that for each $U$
%
\begin{equation}\label{eq.hola_gramola_1}
	\left(\frac{|G|
	}{|\pi_U(S^{(k)})|}\right)^{\frac{1}{|U|-1}}
\geq
	 |\pi_U(S^{(k)})|^{-\frac{1}{|U|}}.
\end{equation}
 Indeed, if
\begin{equation}\label{eq.hola_gramola}
	\left(\frac{|G|
	}{|\pi_U(S^{(k)})|}\right)^{\frac{1}{|U|-1}}
<
	 |\pi_U(S^{(k)})|^{-\frac{1}{|U|}}
\end{equation} for some $U$, then
\begin{equation}\label{eq.auxi_1}
	|\pi_U(S^{(k)})|
> |G|^{|U|},
\end{equation}
which contradicts $|\pi_U(S^{(k)})|\leq |G|^{|U|}$.
 Hence
\begin{equation}\label{eq.gap_hom_0}
\max_{\substack{U\subseteq [1,k]\\ |U|\geq 2}}
	\left(\frac{|G|
	}{|\pi_U(S^{(k)})|}\right)^{\frac{1}{|U|-1}}\geq \max_{\substack{U\subseteq [1,k]\\ |U|\geq 2}} |\pi_U(S^{(k)})|^{-\frac{1}{|U|}}.
\end{equation}
Now observe that, if
\begin{equation}\label{eq.gap_hom_1}
\max_{\substack{U\subseteq [1,k]\\ |U|\geq 2}}
	\left(\frac{|G|
	}{|\pi_U(S^{(k)})|}\right)^{\frac{1}{|U|-1}}\approx \max_{\substack{U\subseteq [1,k]\\ |U|\geq 2}} |\pi_U(S^{(k)})|^{-\frac{1}{|U|}},
\end{equation}
then \eqref{eq.hola_gramola_1} implies that
\begin{equation}\label{eq.hola_gramola_2}
	\left(\frac{|G|
	}{|\pi_{U_0}(S^{(k)})|}\right)^{\frac{1}{|U_0|-1}}
\approx
	 |\pi_{U_0}(S^{(k)})|^{-\frac{1}{|U_0|}}
\end{equation}
for the $U_0$ giving the maximum. Hence by the argument from \eqref{eq.hola_gramola} to \eqref{eq.auxi_1} we obtain
\begin{equation*}
 \max_{\substack{U\subseteq [1,k]\\ |U|\geq 2}} |\pi_U(S^{(k)})|^{-\frac{1}{|U|}}=|\pi_{U_0}(S^{(k)})|^{-\frac{1}{|U_0|}}\approx |G|^{-1},
\end{equation*}
which contradicts the second part of \eqref{eq.hola_gramola_jajaja}. Therefore \eqref{eq.gap_hom_0} and the impossibility of \eqref{eq.gap_hom_1} implies
\begin{equation}\label{eq.gap_hom}
o\left(\max_{\substack{U\subseteq [1,k]\\ |U|\geq 2}}
	\left(\frac{|G|
	}{|\pi_U(S^{(k)})|}\right)^{\frac{1}{|U|-1}}\right)=\max_{\substack{U\subseteq [1,k]\\ |U|\geq 2}} |\pi_U(S^{(k)})|^{-\frac{1}{|U|}}
\end{equation}
holds.
Let us now compare $p$ with the right-hand side of Equation~\eqref{eq.gap_hom}. Assume first that
\begin{equation} \label{e.second_ineq_hom_0}
	o(p)=\max_{\substack{U\subseteq [1,k]\\ |U|\geq 2}} |\pi_U(S^{(k)})|^{-\frac{1}{|U|}}.
\end{equation}
Relation~\eqref{e.second_ineq_hom_0} implies that
\begin{equation}\label{eq.absolute_lower_bound}
o(p)=|S^{(k)}|^{-\frac{1}{k}}.
\end{equation}
Therefore, we can use Proposition~\ref{p.hom_concentration} to show concentration around the mean in this interval. In particular, we can apply Theorem~\ref{t.zero_statement} to obtain the 0-statement in the region defined by (\ref{e.first_ineq_hom_0}) and (\ref{e.second_ineq_hom_0}) as there is an asymptotic gap between the two by (\ref{eq.gap_hom}), and additionally (\ref{eq.absolute_lower_bound}) holds.

Assume now that $$\max_{\substack{U\subseteq [1,k]\\ |U|\geq 2}} |\pi_U(S^{(k)})|^{-\frac{1}{|U|}}=|S^{(k)}|^{1/k}.$$
Then Proposition~\ref{t.zero_statement0} allows to complete the remaining range of probabilities left by Assumption \eqref{e.second_ineq_hom_0}.

In the cases where $\max_{\substack{U\subseteq [1,k]\\ |U|\geq 2}} |\pi_U(S^{(k)})|^{-\frac{1}{|U|}}>|S^{(k)}|^{1/k}$, consider $U_0$ to be the set for which $$\max_{\substack{U\subseteq [1,k]\\ |U|\geq 2}} |\pi_U(S^{(k)})|^{-\frac{1}{|U|}}=|\pi_{U_0}(S^{(k)})|^{1/|U_0|},$$ and consider a new family of systems of configurations $(S,G)=(\pi_{U_0}(S^{(k)}),G)$. To conclude, observe that if $|\pi_{U_0}(S^{(k)})|^{1/|U_0|}=o(|G|)$, then we can apply the reasoning of Proposition~\ref{t.zero_statement0} to this new system of configurations to show the 0-statement (possibly with an adjustment in the constant $c_2$). The case $|\pi_{U_0}(S^{(k)})|^{1/|U_0|}\approx |G|$ cannot occur by (\ref{eq.hola_gramola_jajaja}).
\end{proof}

\subsubsection{Examples of configurations}~\label{ss.example}
So far, authors have found results on the existence
of  configuration in subsets by studying systems of configurations arising from integer linear systems in $[1,n]$ invariant by translations \cite{Sch12,CG10}, integer linear systems over abelian groups \cite{ST12} or linear systems over finite fields \cite{ST12}. These generalize the case for $k$-term arithmetic progressions \cite{BMS14,CG10,Sha10,ST12} (see also further comments in Subsection \ref{subsec:sys_lin_eq}). Here we present some examples that do not seem to follow from the previous results.

All the previous systems of configurations can be seen as prominent particular cases
of homomorphisms of finite abelian groups, context in which an arithmetic removal lemma can be found in \cite[Theorem~2]{Vena14}; these include linear homothetic-to-a-point configurations in products of finite abelian groups that were considered by Tao~\cite{T12}.
 The framework of homomorphisms also includes the configurations from the multidimensional Szemer\'edi setting \cite{Furst77,T12}, some of which have been treated in \cite{BMS14,Sch12}.

The following theorem illustrates an application of Theorem~\ref{t.count_conf_1} which can not be directly obtained form the previously existing tools.

\begin{theorem}[Rectangles in abelian groups]\label{thm:square} Let $\{G_{i}\}_{i\geq 1}$ be a sequence of finite abelian groups, $H_i,K_i$ subgroups of $G_i$ and such that $|H_i|,|K_i|, |G_i|\to \infty$.
For each $\delta>0$ with $\delta<1/40$ there exist $C=C(\delta)$ and $i_0>0$, depending on the family $\{G_i,H_i,K_i\}_{i\geq 1}$ and on $\delta$, for which the following holds. Let
\begin{displaymath}
	S_i=\{(x,x+a,x+b,x+a+b):\, x\in G_i, a\in H_i, b\in K_i\}
\end{displaymath}
be the set of configurations. Assume that $\max\{|H_i|,|K_i|\}\leq(|S_i^{(4)}|/|G_i|)^{2/3}$.
For each $i\geq i_0$ the number of sets free of configurations in $S_i^{(4)}$ and with cardinality $t$ such that
\begin{displaymath}
	t> \frac{C}{\delta}\left(\frac{|G_i|^{4}}{|S_i^{(4)}|}\right)^{1/3}
\end{displaymath}
is bounded from above by ${ 2 \delta |G_i| \choose t}$.
\end{theorem}

Let us remark that Theorem~\ref{thm:square} can be proven, in some cases, using the \cite[Theorem~B.1]{T12}. For instance, let the set of squares in $\Z_i\times \Z_i$ be $S_i=\{\left((x,y),(x+a,y),(x,y+a),(x+a,y+a)\right):\, x,y,a\in \Z_i\}$. By considering all possible ratios between the sides of the square:
$$\left\{\{\left((x,y),(x+ca,y),(x,y+da),(x+ca,y+da)\right):\, x,y,a\in \Z_i\}\;\right\}_{c,d\in \Z_i}$$
and applying \cite[Theorem~B.1]{T12} to each of them, one can obtain the result for rectangles Theorem~\ref{thm:square} when $K_i=H_i=\Z_i$ and $G_i=K_i\times H_i$. This iterative argument works in some very particular cases, for instance when both $K_i$ and $H_i$ grow with $i$, when their intersection is trivial, and there is a certain decomposition of all the configuration in treatable pieces as before (the rectangles are considered as squares with different ratios between the sides). However, this iterative argument cannot cover cases where, for instance, the exponent of either $K_i$ or $H_i$ is bounded.

\begin{proof}[Proof of Theorem~\ref{thm:square}]
First let us observe that the solution set is isomorphic to the group $G_i\times H_i\times K_i$, which is a subgroup of $G_i^3$ (in $G_i^4$).
Indeed,
 $$S_i=\{(x_1,x_2,x_3,x_4):\, x_2-x_1\in H_i,\, x_3-x_1\in K_i,\, x_4-x_1=x_2-x_1+x_3-x_1\},$$
 which can be defined in terms of the kernel of an homomorphism between abelian groups.
	Assume $\max\{|K_i|,|H_i|\}=|K_i|$.
All the hypothesis of Corollary~\ref{o.counting_sol_free_hom} are satisfied, so we shall check that, under the hypothesis $\max\{|H_i|,|K_i|\}\leq(|S_i^{(4)}|/|G_i|)^{2/3}$, then
\begin{displaymath}
	C' \frac{|G_i|}{\delta}  \max_{\ell\in[2,4]}\left\{ \left(\frac{ \alpha^4_{\ell}}{\alpha^4_{1}} \frac{1}{4}{4\choose \ell} \right)^{\frac{1}{\ell-1}} \right\}=\frac{C}{\delta}\left(\frac{|G_i|^{4}}{|S_i^{(4)}|}\right)^{1/3}.
\end{displaymath}
To compute $\alpha_2^4$ we shall compute, for $U\subset [1,4]$ with $|U|=2$, $\max_{(g_1,g_2)\in G_i} |S_i^{(4)}(A_i,G_i)\cap \pi^{-1}_{U}(g_1,g_2)|$. If $U$ is $\{1,2\}$, $\{1,3\}$, $\{2,4\}$ or $\{3,4\}$, then the size of the preimage is, approximately, $|K_i|$ or $|H_i|$ respectively. In the other cases the sizes depends, essentially, on $K_i\cap H_i$, hence they are smaller. Hence $\alpha_2^4\approx \max\{|K_i|,|H_i|\}=|K_i|$. The second equality holds by assumption.

Let us compute $\alpha_3^4$: if all the elements are different, there is always just one choice as given $3$ points of the above rectangle, the 4th point is always computed uniquely. Hence $\alpha_3^4=1$. Therefore, also $\alpha_4^4=1$.
Consequently,
\begin{displaymath}
\max_{\ell\in[2,4]}\left\{ \left(\frac{ \alpha^4_{\ell}}{\alpha^4_{1}} \right)^{\frac{1}{\ell-1}}\right\}
\approx \max\left\{\left(\frac{|G_i|}{|S_i^{(4)}|}\right)^{\frac{1}{3}},\left(\frac{|G_i|}{|S_i^{(4)}|}\right)^{\frac{1}{2}},\frac{|K_i||G_i|}{|S_i^{(4)}|}\right\}.
\end{displaymath}
Since $S_i^{(4)}\geq |G_i|$ and $\max\{|H_i|,|K_i|\}\leq(|S_i^{(4)}|/|G_i|)^{2/3}$, then
\begin{displaymath}
\max_{\ell\in[2,4]}\left\{ \left(\frac{ \alpha^4_{\ell}}{\alpha^4_{1}} \right)^{\frac{1}{\ell-1}}\right\}
\approx \left(\frac{|G_i|}{|S_i^{(4)}|}\right)^{\frac{1}{3}}
\end{displaymath}
and the result follows from Corollary~\ref{o.counting_sol_free_hom}.
\end{proof}
Let us discuss briefly the case when $G_i=\Z_i^2$, $H_i=\Z_i\times \{0\}$, $K_i=\{0\}\times \Z_i$, or more generally when $G_i=H_i \times K_i$ (with $|H_i|=|K_i|=i$).
Then the size $t$ in Theorem \ref{thm:square} satisfies that $t>C'' i^{4/3}$, for a certain constant $C''$.
Observe that a set $X \subset G_i$ defines a bipartite graph in the following way: for each $(h,k)\in G_i=H_i\times K_i$ construct the bipartite graph with vertex set $V=H_i \cup K_i$ by adding an edge between $h$ and $k$.
If $X$ is solution-free, then the corresponding bipartite graph does not contain cycles of length $4$.
In this situation, it is well known that the maximum number of edges in a bipartite graph (with bipartition of size $i\times i$) without cycles of length four
is  $i^{3/2}+o(i^{3/2})$
as a consequence of a classical result of K\"ovari, S\'os and Tur\'an~\cite{KoSoTu54}.
This shows that for $C i^{4/3}<t<c i^{3/2}$, there is a wide range of values of $t$ to which Theorem~\ref{thm:square} gives a meaningful bound.

Theorem~\ref{thm:square} can be generalized to higher dimensional cubes or to other configurations in a two-dimensional cartesian product with more points.
More precisely, given $\{G_i\}_{i\geq 1}$ a sequence of finite abelian groups and $K_i$, $H_i$ two subgroups (whose size grows when $i\rightarrow \infty$), we can consider configurations $S_i$ of the form
\begin{equation}\label{eq:gen.rectangles}
S_i=\left\{(x,x+a_1,\dots,x+a_r,x+b_1,x+b_1+a_1,\dots, x+b_1+a_r,\dots,x+b_r+a_r):\, x\in G_i, a_j\in H_i, b_j\in K_i\right\},
\end{equation}
which generalizes the configuration studied in Theorem \ref{thm:square}.
The same techniques used to prove Theorem~\ref{thm:square} shows that for $0<\delta<1/40$ and assuming that
\begin{displaymath}
	\max\{|H_i|,|K_i|\}\leq \left(\frac{|G_i|}{|S_i^{(k)}|}\right)^{\frac{r+1}{(r+1)^2-1}},
\end{displaymath}
then there exists a constant $C$, that depends on $\delta$ and in the family $\{G_i,H_i,K_i\}_{i\geq 1}$, and an integer $i_0$ such that the following holds: for $t$ such that
\begin{equation}\label{e.gen-zaran}
	t> \frac{C}{\delta}\left(\frac{|G_i|^{(r+1)^2}}{|S_i^{(k)}|}\right)^{\frac{1}{(r+1)^2-1}},
\end{equation}
then the number of solution-free sets of $G_i$ size $t$ (for $i> i_0$) is bounded from above by ${ 2 \delta |G_i| \choose t}$. Observe also that this result can be even more extended when dealing with an asymmetric version, namely considering $r+1$ elements in $H_i$ and $s+1$ elements in $K_i$.

Similarly as for Theorem~\ref{thm:square}, when we particularize \eqref{eq:gen.rectangles} to $G_i=H_i\times K_i$, $H_i \cap K_i=\{0\}$ and $|H_i|=|K_i|=i$, then each set $X \subset G_i$ without solutions in $S_i$ defines a bipartite graph with $|H_i|$ vertices on each stable set without $K_{r+1,r+1}$ as a subgraph. Even more, each point of the type $x+a_i+b_j$ corresponds to an edge, hence we can obtain results for any bipartite graph. However, in this case, the lower bounds such as \eqref{e.gen-zaran} are more involved and depend heavily on the particular system of configurations considered.

This connects the configurations codified by \eqref{eq:gen.rectangles} with the classical Zarankiewicz problem \cite{KoSoTu54}. The problem is to study the function $\mathrm{zex}(n, K_{t,s})$ counting the largest number of edges in a bipartite graph with $n$ vertices on each stable set which excludes $K_{t,s}$ $(s\geq t)$ as a subgraph.
Some partial results are known for the general Zarankiewicz problem: upper bound of $O(n^{2-1/t})$ was obtained by K\H{o}v\'ari, S\'os and Tur\'an in~\cite{KoSoTu54}.
It is conjectured that this upper bound gives the correct order, but the problem of finding lower bounds (namely, explicit constructions) has been shown to be more difficult.
Some instances are known to close the gap in the order of magnitude:
  Erd\H{o}s, R\'enyi and S\'os \cite{ErdReSo66} found a lower bound for $\mathrm{zex}(n,K_{2,t})$ which match the upper bound, 
Brown \cite{Bro66} and F\"uredi \cite{Fur96}  proved the right order of magnitude for $\mathrm{zex}(n,K_{3,3})$. Finally the case $\mathrm{zex}(n,K_{s,t})$ with $s\geq (t-1)!+1$ was proved by Alon, R\'onyai and Szab\'o \cite{AlRoSz99}.

Continuing in the bipartite graph case,
 Balogh and Samotij showed in \cite{balsam11} that the $\log_2$ of the total number of subgraphs on $n$ vertices without a $K_{s,t}$ is bounded above by $c  n^{2-1/t}$ with an explicit constant that depends on $s$ and $t$. The upper bound from \cite{balsam11} is more accurate than what it can be obtained with Theorem~\ref{thm:square} (or its equivalent for $K_{r+1,r+1}$).
For bipartite graphs on $n$ vertices, one can obtain the V-property using the known cases of Sidorenko's conjecture \cite{sid93}, which states that, in a graph with edge-density $d$, there are $d^{e(H)} n^{e(H)}$ graph homomorphisms from a bipartite graph with $e(H)$ edges. Sidorenko's conjecture is known to hold for complete bipartite graphs (see \cite{szeg14}),  hence giving much better bounds than \cite[Theorem~1]{Vena15}.

Let us mention still another generalization of Theorem \ref{thm:square} by exploiting the homomorphism setting.
Let $G$ be a finite abelian group, $H$ a subgroup of $G$ and $\phi: H\to G$ an injective group homomorphism with $a\neq\pm \phi(a)$ for each $a\in H$.
We consider configuration set of `slanted squares' defined by $\{(x,x+a,x+\phi(a),x+a+\phi(a)) :
\; x\in G, a\in H\}$, which includes the rhombuses $\{((x,y),(x+a_1,y+a_2),(x+a_2,y+a_1),(x+a_1+a_2,y+a_1+a_2)):\,x,y,a_1,a_2\in \Z_i\}$.
Other geometric structures that we may consider are isosceles triangles where the uneven side is located along the $x$-axis $\{((x,y),(x+a_1,y+a_2),(x-a_1,y+a_2)):\,x,y,a_1,a_2\in \Z_i\}$, or all the possible right-angled triangles $\{((x,y),(x+a_1,y+a_2),(x+a_2,y-a_1)):\,x,y,a_1,a_2\in \Z_i\}$.
\subsection{Configurations in $[1,n]^m$} \label{s.cube}
In this section we consider linear configurations in $[1,n]^m$.
These linear configurations arise from a group homomorphisms $M:[\Z^m]^k\to [\Z^m]^k$ with the invariant property (namely, $M(x,\ldots,x)=0$ for each $x\in\Z^m$, see Subsection \ref{s.hom}).
These homomorphisms can be also described as in Section~\ref{ss.shape}. Even though
$[1,n]^m$ is not a group, we shall see that the proof of Corollary~\ref{o.counting_sol_free_hom} can be adapted to obtain an analogous result.
Let $S_n=M^{-1}(0)\cap ([1,n]^m)^k$ denote the set of configurations and
$S_n^{(k)}$ the sets of configurations where all the points are different.

Let us make two important remarks which hold as $M$ is linear, with respect to multiplication by integers, and invariant.
\begin{remark}\label{r.1}
If $(x_1,\ldots,x_k)\in S_n$, $x_i\in [1,n]$, then
$(x_1+x,\ldots,x_k+x)\in S_n$ for some $x\in[-n,n]^m$ with sufficiently small coordinates so that $x_i+x\in[1,n]^m$ for each $i$.
\end{remark}
\begin{remark}\label{r.2}
$(\lambda x_1,\ldots,\lambda x_k)\in S_n$ for some $\lambda\in [0,n]$ sufficiently small so that  $\lambda x_i\in [0,n]$ for each $x_i$.
This makes that $\alpha_1$ is achieved at the point $\mathbf{0}\in [1,n]^m$.
\end{remark}
Let us first show a proposition that represents one of the main differences with respect to Corollary~\ref{o.counting_sol_free_hom}.
\begin{proposition}[$\alpha_i$ in the $m$-dimensional cube] \label{p.alf_1_cube}
Given an invariant homomorphism $M:[\Z^m]^{k}\to [\Z^m]^{k}$, and the sequence of systems of configurations $\{(S_i,[1,i]^m)\}_{i\geq 1}$, with $S_i=M^{-1}(0)\cap  ([1,i]^m)^k$, satisfying the V-property with $\gamma_{\text{\ref{l.V-property_invariant_hom}}}$ depending on $M$ and uniform for the family, there exist a $\lambda>0$ and a constant $c>0$, both depending on $M$, such that
\begin{equation}
	\alpha_i(S_n,[1,n]^m) \leq \alpha_i(\overline{S}_{\lambda n},\Z_{\lambda n}^m) \leq c \alpha_i(S_n,[1,n]^m),
\end{equation}
where $\overline{S}_{\lambda n}$ is the kernel of the natural restriction of $M$, a matrix of integers, to $M_n:[\Z_{\lambda n}^m]^{k}\to [\Z_{\lambda n}^m]^{k}$.
\end{proposition}
\begin{proof}
	Since $M$ is a fixed matrix of integers, there exists a $\lambda\in \Z^+$ such that the following holds. Given $M_{\lambda i}:[\Z_{\lambda i}^m]^{k}\to
	[\Z_{\lambda i}^m]^{k}$, the natural restriction to $\Z_{\lambda i}^m$ of $M$, then $\mathbf{x}\in \overline{S}_{\lambda i} \cap ([1,i]^m)^k$ if and only if
	$\mathbf{x}\in S_i$ (this shows $\alpha_i^k(S_n,[1,n]^m) \leq \alpha_i^k(\overline{S}_{\lambda n},\Z_{\lambda n}^m)$). That is, we obtain $\lambda=\lambda(M)$ (the $\lambda$ of the statement) as the minimal value for which, if all the variables have coordinate values in the first $i$ positive integers, then any equation can be read in the integer setting and not in the cyclic group setting.
Now we apply Lemma~\ref{l.V-property_invariant_hom} on $M_{\lambda i}$ for the set $B_i=[1,i]^m\subset \Z_{\lambda i}^m$ as
$|B_i|= \frac{|\Z_{\lambda i}^m|}{\lambda^m}$ (thus $B_i$ represents a positive proportion of $\Z_{\lambda i}^m$).
By the choice of $\lambda$, $\prod_{j=1}^k B_i^k \cap \overline{S}_{\lambda i}= S_i$. This shows the first part of the result.

To prove the second part we observe that, as in the proof of Proposition~\ref{p.hom_concentration}, when we predetermine some of the values for the variables for the system $M_n$, the number of solutions is either $0$ (if the predetermined values renders the system incompatible), or the same number of solutions as in the homogeneous case. In particular, if we select some predetermined values that can be completed to a full solution, the number of solutions projected only depends on the indices of the variables selected. By the first part of the statement already proved, and an averaging argument, we conclude that, for each $U\subset [1,k]$ there exists a $c=c(M)$ and a solutions with all the variables in $[1,n]^m$ with the following property: the number of solutions projected to it when from $S_n$ is a constant $c$ away from the ones projected from $\overline{S}_{\lambda n}$. Therefore, for every $i$, $\alpha_i(\overline{S}_{\lambda n},\Z_{\lambda n}^m) \leq c \alpha_i(S_n,[1,n]^m)$. \end{proof}

\begin{observation}\label{o.a-to-ak-box}
	If instead of $\alpha_i$ we consider $\alpha_i^k$ the results will be similar as if $x_i\neq x_j$ in the given subsystem and we impose the equation $x_i=x_j$, the difference in terms of the sizes of the solution sets is, at least $n$. Therefore we have that $\alpha_i^k\approx \alpha_i$ in this case.
\end{observation}

By Observation~\ref{o.a-to-ak-box}, Proposition~\ref{p.alf_1_cube} and assuming the normality of the family of homomorphism systems induced by $M$, we obtain analogous results to
Proposition~\ref{p.hom_concentration},
Theorem~\ref{t.ramdom_sparse_hom} and Corollary~\ref{o.counting_sol_free_hom}.
Therefore, we obtain counting and random-sparse-analogue results for sets in $[1,n]^m$ free of $m$-dimensional simplices
\begin{equation}\label{eq.simplices}
\{((x_1,\ldots,x_m),(x_1+a,\ldots,x_m),\ldots,(x_1,\ldots,x_m+a))\; |\;x_i,x_i+a\in [1,n]\}
\end{equation}
(multidimensional Szemer\'edi), or other homothetic-to-a-point linear structures. These results have also been obtained in \cite{BMS14,ST12}.

In some cases, we want to consider configurations systems where some of the $a_i$ from Proposition~\ref{p.solution_structure}
are non-negative (such as when we ask in \eqref{eq.simplices} for $a\geq 0$). By considering symmetric configurations (configuration containing the vectors with $+a_i$ and $-a_i$ for every $i$ in Proposition~\ref{p.solution_structure}), we obtain the V-property for these restricted configurations. Furthermore, the $\alpha_i$ in the restricted case are, up to a multiplicative constant, equal to their unrestricted counterparts. Indeed, the total number of solutions is, up to the factor $(1/2)$ raised to the power given by the number $a_i$'s asked to have a specific sign. Similarly as in the proof of Proposition~\ref{p.alf_1_cube}, an averaging argument shows that, up to a multiplicative constant depending on the number of coordinates and on the number of points in the configuration, the maximum number of solutions projected to a partial solution with $a_i$ restricted in sign is the same as in the case of the homomorphism that has been restricted.
\subsection{Linear systems of equations on abelian groups}\label{subsec:sys_lin_eq}
In this subsection we study in detail a prominent case of group homomorphism.
Let $G$ be an abelian group.
Following the language of Module Theory, for $g\in G$, $n\in \mathbb{N}$, we define $ng=g+\stackrel{(n)}{\dots}+g$ (the definition can be extended to negative $n$).
Let $A$ be $k\times m$ matrix with integer entries, $(x_1,\dots, x_k)\in G^{k}$ and $\textbf{x}=(x_1,\dots, x_k)^{\mathrm{T}}$.
We consider the group homomorphism from $G^{k}$ to $G^{m}$ defined by the matrix multiplication $(x_1,\dots, x_k) \mapsto A \textbf{x}$.
This homomorphism defines a system of configurations whose solutions are the elements $\textbf{x}\in G^k$ such that $A \textbf{x}=0$.
As usual, we refer to these type of homomorphisms as linear system of equations.
In all this section we assume that $A$ has maximum rank.

We start by discussing results in the integer scenario in~\ref{ss.integersystems}, where we relate our framework with~\cite{rodruc97}.
Later, we develop on the case when dealing with finite fields and finite abelian groups in~\ref{ss.linear}.
Let us mention that similar results appeared in the work of Saxton and Thomason in~\cite{ST12}.
\subsubsection{The integer case}\label{ss.integersystems}
Following~\cite{rodruc97}, we assume that $A$ is \emph{irredundant}, namely for each pair of indices $i\neq j$, there exist a solution $(x_1,\ldots,x_k)$ with $x_i\neq x_j$.
In particular, irredundancy implies that $S^{(k)}\subset A^{-1}(0)$ is non-empty.
This naturally relates with the condition of \emph{abundancy} introduced in \cite{ST12}.
Indeed, an abundant matrix is irredundant, but not every irredundant matrix is abundant. For instance $x_1-2 x_2=0$ is irreducible but not abundant.
See for example \cite{RuZU14} for a wide variety of explicit instances studied in the literature fitting with this setting.

We are also interested when taking coordinates in the interval $[1,n]$ (or in some cases in $[-n/2,n/2]$) instead of $\Z$ to obtain quantitative results.
Obviously in this case we do not have a group structure (and hence, a group homomorphism).
However, as discussed in Subsection~\ref{s.cube} we can also study this situation without technical problems (when the same fixed integer matrix is considered for all $n$), and we obtain the same properties assured by the condition of invariant homomorphism (as we can place $[1,n]$ in a convenient cyclic group $\Z/s\Z$, for $s$ large enough depending on $A$ to univocally transfer the solution sets between the two settings).

We say that $A$ satisfies the \emph{strong column condition} if the sum of the columns of $A$ is zero.
This is equivalent to the fact that $A$ induces an invariant homomorphism.
Frankl, Graham, R\"odl \cite{FGR88} (see also \cite[Theorem 6.1]{SV13}) proved the V-property in this particular case (which the authors call \emph{density regular} property).
More precisely, for every $\delta>0$ there exists $n_0=n_0(\delta)$ and an $\epsilon=\epsilon(\delta,A)>0$ with the following property: for every  $n\ge n_0$ and for every set $X\subset [1,n]$ with $|X|\geq \delta n$, the linear system $A\textbf{x}=0$ satisfying the strong columns condition has at least $ \epsilon n^{k-m}$ solutions with $\mathbf{x}\in A^{-1}(0)\cap [1,n]^k=S^{(k)}$. Moreover, if the matrix $A$ does not satisfy the strong column condition then $A$ is not density regular.
Indeed, one can obtain that solutions given by the V-property could be assume to have pairwise different components (see for instance \cite{RuZU14}).
Complementarily, using the fact that $A$ is irredundant we can deduce that $\xi_{\text{\ref{t.count_conf_1}}}/|S^{(k)}|>a>0$, for some $a$ depending on $A$ and on $\delta$, is satisfied and Theorem \ref{t.count_conf_1} applies in this setting.
Irredundancy implies that the system $A'$ formed by $A$ when any equation $x_i=x_j$ is added reduces the rank by $1$ while keeping the same number of variables. Since the number of variables is bounded by $k$, we obtain that  $|A'^{-1}(0)\cap [1,n]^k|=O(n^{k-m-1})$ hence $|S^{(k)}|\approx |S|\approx n^{k-m}$.

Let us compute the bounds for the parameter $t$ in Theorem \ref{t.count_conf_1}.
In fact, for an irredundant system $A$, such parameter was already defined as $m_A$ by R\"odl and Ruci\'nski in \cite[Definition~1.1]{rodruc97}:
\begin{equation}\label{eq.rodlRuck}
	m_A=\max_{q\in[2,k]}\max_{\substack{B\subset [1,k]\\ |B|=q}}
	\frac{q-1}{q-1+h_{B}-\ell}
\end{equation}
where $\ell$ is the rank of the matrix $A$, $h_{B}$ stands for the rank of the matrix $A^{B}$ (namely, the submatrix of $A$ where the columns indexed by $B\subset [1,k] $ have been deleted).
Then, the right order of magnitude for the lower bound for $t$ is $n^{1-1/m_A}$.
Let us also mention that this parameter $m_A$ was already exploited in the work of Schacht \cite{Sch12}, Friedgut, R\"odl and Schacht \cite{FriRodSch10}, and Saxton and Thomason \cite{ST12}.
We now show that our specification of $t$ in terms of the different coefficients $\{\alpha_l^k\}_{1\leq l\leq k }$ gives exactly $m_A$.

Fix $1\leq l \leq k$.
Observe first that, from linearity of the equation $A\mathbf{x}=0$ and the fact that $A$ is irredundant, $\alpha_l$ and $\alpha_l^k$ have the same order of magnitude: when fixing $l$ coordinates of $\textbf{x}$ in positions indexed by $B \subset [1,k]$, $l=|B|$, the new system of equations $A^B \textbf{y}=\textbf{b}$ has either $0$ solutions or a number that only depends on the rank of $A^B$.
Assume that there exists a solution $\textbf{x}_0=(\textbf{x}_{\overline{B}},\textbf{x}_B)$ with all coordinates being different such that $A^B \textbf{x}_{B}=-A^{\overline{B}} \textbf{x}_{\overline{B}}=\textbf{b}$.
Each solution $\textbf{y}_{B}\in S(A^{B},\Z)$ of
$A^B \textbf{y}_{B}=0$ for which  $\textbf{y}_{B}+\textbf{x}_{B}$ has some repeated coordinates, say indexed by $i$ and $j$, must satisfy an additional equation of the type $y_{B,i}-y_{B,j}=x_{B,j}-x_{B,i}=d\neq 0$; for $A^{B}$ irredundant for the pair $(i,j)$, the previous additional constrain implies lowering the rank by $1$ thus meaning asymptotically less solutions (lowered by a factor of $1/n$); for $A^B$ not being irredundant for the pair $(i,j)$, the addition of the equation $y_{B,i}-y_{B,j}=d$ makes the new system incompatible, so no solution added to $\textbf{x}_B$ will equate the coordinates $i$ and $j$.
Hence, most (asymptotically all) of the solutions that project to $\textbf{x}_{\overline{B}}$ have all different coordinates.
An analogous reasoning can be applied when there is no solution  $\textbf{x}_0=(\textbf{x}_{\overline{B}},\textbf{x}_B)$ with all coordinates being different such that $A^B \textbf{x}_{B}=-A^{\overline{B}} \textbf{x}_{\overline{B}}=\textbf{b}$ to observe that the orders of magnitude of the number of solutions that are projected to $\textbf{x}_{\overline{B}}$ are the same regardless of whether all the coordinates are asked to be different or no.
 Therefore, by taking the maximum over all $B$ of fixed size $l$ we conclude that $\alpha_l$ and $\alpha_l^k$ must have the same order of magnitude. Compare this argument with the one in the general setting of homomorphism developed in the proof of Proposition~\ref{p.hom_concentration} (See the argument for $\Delta_2$ in the proof of Proposition~\ref{p.hom_concentration}).

We can study then the modified parameter
\begin{equation} \label{e.def_param}
 \max_{l\in[2,k]}\left\{ \left(\frac{ \alpha_{l}}{\alpha_1}\right)^{\frac{1}{l-1}} \right\}=	c\max_{l\in[2,k]}\left\{ \left(\frac{ \alpha_{l}}{\alpha_1} \frac{1}{k}{k\choose l} \right)^{\frac{1}{l-1}} \right\}
\end{equation}
for some $c$ depending on $k$.
Observe that left hand side of \eqref{e.def_param} is the value appearing in Theorem~\ref{t.count_conf_1} but where $\alpha_l^k$ have been replaced by $\alpha_l$.
When the solutions of $A\mathbf{x}=0$ are restricted with $\mathbf{x}\in [1,n]^k$ or ($\mathbf{x}\in [-n/2,n/2]^k=G^k$), we have that $|A^{-1}(0)\cap [1,n]^k|=c_A n^{k-\ell}$. Let $S_n=A^{-1}(0)\cap [1,n]^k$ and
consider
$$\alpha_i=\max_{\substack{B\subset[1,k]\\|B|=i}} \max_{(g_1,\ldots,g_i)\in G^i} \left\{|S_n\cap \pi_B^{-1}(g_1,\ldots,g_i)|\right\}.$$
Observe that, fixing $B\subset[1,k]$ with $|B|=i$ then
\begin{displaymath}
	\max_{(g_1,\ldots,g_i)\in G^i} \left\{|S_n\cap \pi_B^{-1}(g_1,\ldots,g_i)|\right\}=c_{A,B,i}\; n^{k-i-h_B},
\end{displaymath}
as $k-i-h_B$ are the degrees of freedom: the difference between the free variables, $k-i$, minus the rank of the matrix, or the number of relations/valid equations between the variables.
Therefore,
\begin{displaymath}
	\alpha_i=\max_{\substack{B\subset[1,k]\\|B|=i}} \left\{ c_{A,B,i}\; n^{k-i-h_B}\right\}=
	c'_{A,i} \max_{\substack{B\subset[1,k]\\|B|=i}} n^{k-i-h_B}.
\end{displaymath}
It is not difficult to see that $\alpha_1=c_A n^{k-1-\ell}$ as, if the matrix is irredundant, there is a variable for which, if we fix it, the rank of the new matrix (with one less column), does not change.
Indeed, if the matrix is irredundant, $m\geq k+1$ as there is at least one non-zero solution to $A\mathbf{x}=0$.
If the matrix is full rank there is one $k\times k$ full rank submatrix. The claim follows.
We conclude substituting everything in \eqref{e.def_param} that
\begin{align}
	\max_{i\in[2,k]}\left\{ \left(\frac{ c'_{A,i} \max_{\substack{B\subset[1,k]\\|B|=i}} n^{k-i-h_B}}{c_A n^{k-1-\ell}}\right)^{\frac{1}{i-1}} \right\}
& = c'''_{A}
\max_{i\in[2,k]}\left\{ \left(
\max_{\substack{B\subset[1,k]\\|B|=i}} n^{k-i-h_B-k+1+\ell}\right)^{\frac{1}{i-1}}\right\} \nonumber \\
&=
\max_{i\in[2,k]}
\max_{\substack{B\subset[1,k]\\|B|=i}}
n^\frac{-i-h_B+1+\ell}{i-1} =
\max_{i\in[2,k]}
\max_{\substack{B\subset[1,k]\\|B|=i}}
n^\frac{-i-h_B+1+\ell}{i-1}. \nonumber
\end{align}

Hence the above quantity is maximal whenever
$\frac{-i-h_B+1+\ell}{i-1}$ is maximal on the appropriate domain. But
$\frac{-i-h_B+1+\ell}{i-1}$ is maximal if and only if
\begin{displaymath}
	-\frac{1}{\frac{-i-h_B+1+\ell}{i-1}}=\frac{i-1}{i+h_B-1-\ell}
\end{displaymath}
is maximal, which is precisely the quantity $m_A$.
Hence, we have found the relation between $\{\alpha_l^k\}_{1\leq l \leq k}$ (which havehave the same order of magnitude as $\{\alpha_l\}_{1\leq l \leq k}$) and $m_A$.
Putting all together we get the following theorem (which extends  \cite[Theorem 1.1]{BMS14}):

\begin{theorem} \label{thm_systems}
Let $A$ be a $k\times m$ irredundant matrix, $k>m$,  with integer entries and maximum rank.
For every positive $\beta$ there exists constants $C=C(A,\beta)$ and $n_0=n_0(A,\beta)$ such that if $n\geq n_0$ and $t\geq C n^{1-1/m_A}$, then  the number of solution--free subsets of size $t$ of $[1,n]$ to the system of equations $A\textbf{x}=0$  is at most
$$
\binom{\beta n}{t}.
$$
\end{theorem}

Let us see a direct consequence of Theorem~\ref{thm_systems}.
Denote by $ex(A,n)$ the size of the largest subset $F\subset [1,n]$ which contains no non-trivial solution of the equation $A\mathbf{x}=0$.
By a trivial solution, following Ruzsa~\cite{R93} and Shapira~\cite{Sh06}, we mean a solution which has constant value on variables whose coefficients add up to zero.
Computing $ex (A,n)$ is not obvious and depends heavily on $A$.
For linear equations ($m=1$) the situation can be illustrated as follows.
Let $L=(a_1,\ldots,a_k)$. Ruzsa names the linear equation $L\cdot \textbf{x}=0$ to be of \emph{genus} $g$ if $g$ is the size of the largest partition of the coefficient set such that the sum of coefficients in each part is zero.
In \cite{R93} he proves that $ex(L,n)\ll n^{1/g}$.

Theorem~\ref{thm_systems} presents an upper bound for the number of solution-free subsets of size $t$ with $t\geq C n^{1-1/m_A}$ (notice that $m_A>0$, hence $1-1/m_A<1$). For equations of genus $g>0$, the size of the largest solution-free set is of the order  $n^{1/g}$. Hence, we are interested in the cases when $1/g> 1-1/m_A$.
It is known that the number of linear systems equations for which the extremal free sets can be (almost)-linear is not negligible. Furthermore,
Shapira \cite{Sh06} shows that almost all linear systems of equations satisfying the strong columns condition have sharp sublinear (Behrend--type) examples which are solution--free.
More precisely, let $\A (k,m,h)$ be the set of $k\times m$ matrices $A$ with integer coefficients such that
\begin{itemize}
\item[(i)] all coefficients in $A$ are bounded in absolute value by $h$,
\item[(ii)]
$A$ satisfies the strong columns condition, so the coefficients of every
row in $A$ sum to zero.
\item[(iii)] $m\le k-\lceil \sqrt{2k}\rceil +1$.
\end{itemize}
\begin{theorem}[Shapira \cite{Sh06}]\label{t.shapira}  Let $k\ge 6$ and $h$ be positive integers. There are $c(k,h)$ and $c(k)$ such that all but at most $c(k)/h$ of the matrices in $\A (k,m,h)$ satisfy a Behrend--type lower bound:
$$
ex(A,n)\ge ne^{-c(k,h)\sqrt{\log n}}.
$$
\end{theorem}
Roughly speaking, Theorem \ref{t.shapira} tells that the large majority of system of equations have a Behrend--type lower bound for $ex(A,n)$.
Consequently, Theorem \ref{thm_systems} gives (for almost all systems) a non-trivial tight bound for the number of solution-free sets in the regime
$$n^{1-1/m_A}\leq t \leq ne^{-c(A)\sqrt{\log n}}. $$
\begin{remark} \label{r.syst}
 Saxton and Thomason prove similar results for linear systems of equations using the container methodology in \cite{ST12}, which slightly differs from \cite{BMS14}. See Subsection \ref{ss.linear} for a detailed explanation of their results.
\end{remark}	
Let us finally discuss the random sparse counterpart.
The arguments for invariant homomorphisms in Subsection \ref{s.hom} applies in this framework and Definitions \ref{d.concentration_intersection} and \ref{d.g-uniform} applies.
Moreover, the previous arguments had shown that, for each $\delta>0$ the threshold probability
for

the $(\delta, A^{-1}(0)\cap [1,n]^k)_k$-stability is $p_A=n^{-1/m_A}$.
This result was first obtained by Schacht \cite[Theorem 2.4]{Sch12} and Saxton and Thomason \cite[Theorem 12.3]{ST12}, and extends the sparse Szemer\'edi--type of Conlon and Gowers \cite[Theorem 1.12]{CG10}.
\subsubsection{Linear system of equations in other groups}\label{ss.linear}
In this subsection we discuss results for linear system of equations over finite fields and abelian groups.
In both cases the V-property holds as a consequence of the V-property for group homomorphisms, and also from previous works of Kr\'al', Serra and Vena \cite{KRSV12,KrSV13}.

In the finite field setting, the computations are essentially the same as in the integer case, giving rise to the very same constant $m_A$ as defined in Equation ~\eqref{eq.rodlRuck}.
Hence, this gives the analogue of Theorem~\ref{thm_systems} for equations over finite fields.
For finite fields we could consider matrices with coefficients over the field and the results would also be analogous.
Again, in this setting we have lower bounds for the size of sets $X \subset \mathbb{F}_p^n$ avoiding solutions to a given equation $A \mathbf{x}=0$:
when $n$ is fixed and $p$ tends to infinity, Behrend-type constructions from integers transfer easily to constructions in $\mathbb{F}_p^n$.
However, the case with $p$ fixed and $n$ tending to infinite had been less studied and there are only few families studied in this context, see for instance \cite{LinWolf2010}.

Let us finally shortly describe this setting over general abelian groups, which slight differs from the previous cases. In this context, as pointed out in \cite{ST12}, the rank of a matrix over an abelian group is not well defined, and hence Equation \eqref{eq.rodlRuck} must be computed by other sources. See \cite[Section 10]{ST12} for the details of the right parameter $m_A$ in this context. Let us also mention that the \emph{determinantal} condition needed in \cite[Theorem 10.3]{ST12} is not necessary due to the group homomorphism result we are using, that extends~\cite[Theorem~1]{KrSV13}.
%
%
\subsection{Configurations in non-abelian groups}\label{ss.nonabelian}
In this subsection we discuss still another family of examples arising from equations on non-abelian groups.
To simplify notation, we write $e$ for a generic identity element on a group $G$.
The main theorem we can prove is the following:

\begin{theorem} Let $r_1,\ldots ,r_k$ be fixed positive integers and $r=r_1+\cdots+r_k$. Let $\{G_i\}_{i\ge 1}$ be a sequence of groups with unit element $e$. Assume that the exponent of $G_i$ is a divisor of $r$ and that for every $j$, $\gcd (r_j,|G_i|)=1$. Then, for each $\delta >0$ with $\delta\le \min \{\beta/2,1/40\}$ there exist positive constants, $i_0$, $c_1$ and $c_2$, $C$ such that the following hold: for each $i\geq i_0$ and $t$ in the margin
$$
 \frac{C }{\delta}k^{-\frac{1}{k-1}}  |G_i|^{\frac{1}{k-1}}\leq t \leq \frac{\delta}{2}|G_i|,
$$
there are at most
$$
{\beta|G_i|\choose t}
$$
sets in $G_i$ which are free of solutions to the equation
\begin{equation}\label{eq.de_les_ult}
x_1^{r_1}\cdots x_k^{r_k}=e.
\end{equation}
Furthermore, given $S_i$ the solution set induced by \eqref{eq.de_les_ult} and $p_{(S_i,G_i)}=|G_i|^{-\frac{k-2}{k-1}}$ then
\begin{displaymath}
	 \lim_{i\to \infty}
\mathbb{P} ([G_i]_p\,\,\mathrm{is}\,\, (\delta,S_i)_k\text{-stable})=
\left\{
\begin{array}{cc}
	1 & \text{if } p\geq c_1 p_{(S_i,G_i)}, \\
	0 & \text{if } p< c_2 p_{(S_i,G_i)}.\\
\end{array}
\right.
\end{displaymath}
\end{theorem}

\begin{proof}
%
Define $S_i=\{(x_1,\dots, x_k)\in G_i^k: x_1^{r_1}\dots x_k^{r_k}=e\}.$
Observe that such equation does \emph{not} emerge from any group homomorphism, and that the group setting is general (the groups are not necessarily abelian).
This system of configurations satisfies the conditions of our framework: as it was shown by Kr\'al', Serra and Vena in \cite{Krsv09}, for each $\delta>0$, $(S_i,G_i)$ satisfies the V-property with a certain function $\gamma_{\text{\cite{Krsv09}}}(\delta,k)$.
As $|S_i|=|S_i^{(k)}|(1+o(1))$, it is obvious that $(\gamma_{\text{\cite{Krsv09}}}(\delta,k)-1)|S_i|+|S_i^{(k)}|>a|S_i^{(k)}|>0$ (for some $a$ depending on $\delta$ and on $k$), hence $\xi_{\text{\ref{t.count_conf_1}}}/|S_i^{(k)}|>a>0$ is satisfied and we are under the assumptions of Theorem~\ref{t.count_conf_1} (note that it is also true that every subset of size greater than $\delta |G_i|/2$ contains a configuration in $S_i^{(k)}$).

We can easily compute the parameters $\alpha_i$.
Note that as $\gcd(r_s,|G_i|)=1$ then the function $f_r:G_i\rightarrow G_i$, with $f_{r_s}(g)=g^{r_s}$ is a bijection.
Hence, when fixing a set of $s$ variables on the equation $x_1^{r_1}\cdots x_k^{r_k}=e$, then number of solutions is equal to $|G_i|^{k-i-1}$.
Consequently $\alpha_i=|G_i|^{k-i-1}$ and $\alpha_i^k=|G_i|^{k-i-1}(1+o(1))$.
The last equality holds because the number of solutions with $i$ fixed components and repeated variables is $o(|G_i|^{k-i-1})$.
Finally, for $i=k$ we have $\alpha_k=\alpha_k^{k}=1$ (a single equation has always solutions with all different components).
Indeed, these arguments hold due to the fact of dealing with a single equation.
Hence, the margin for $t$ for which Theorem~\ref{t.count_conf_1} applies is
\begin{equation}\label{eq.t-nonabelian}
 \frac{C }{\delta}k^{-\frac{1}{k-1}}  |G_i|^{\frac{1}{k-1}}\leq t \leq \frac{\delta}{2}|G_i|,
\end{equation}
where $C$ is the constant stated in Theorem~\ref{t.count_conf_1}.

Finally, let us consider now the random counterpart. In the particular case of a single equation it is straightforward to show both the concentration and the uniformity properties, namely Definitions~\ref{d.concentration_intersection} and~\ref{d.g-uniform} (indeed, one can argue exactly in the same way as in the case of a single equation on the abelian setting, which is covered by the group homomorphism setting). This gives a threshold probability function equal to $p_{(S_i,G_i)}=|G_i|^{-\frac{k-2}{k-1}}$.
\end{proof}
\begin{remark}
As it is shown, in \cite{Krsv09}, the V-property is also satisfied for certain system of equations in non-abelian groups which are \emph{graph representable} (see \cite[Section 3]{Krsv09}). So a similar analysis could be also done for the corresponding system of configurations.
\end{remark}
%

\section{Further research} \label{s.non_ab}

In this paper we have provided a wide variety of examples in which we can combine the hypergraph container technique joint with supersaturation results arising from removal lemmas in different scenarios.
Let us mention that families arising from non-linear configurations (in which a V-property also exists) can be studied as well. This covers the polynomial extension of Szemer\'edi Theorem due to Bergelson and Leibman~\cite{BerLei96} (see also~\cite{Gre02}).

On the other side, there are configurations which still fall beyond the reach of our methods. Let us mention a couple of them.
First, let us discuss an example very similar in shape to the one studied in Theorem~\ref{thm:square}.
In \cite{FurHaj92} the authors study the maximum number of $1$ in a $0-1$ $n \times n$ matrix without certain configurations (by configuration we mean a given partial matrix with 1's and blanks at the entries: there is a certain ordering in the position of the 1's).
The problem slightly differs from the study of the maximum number of edges in a complete bipartite graph $K_{n,n}$ without copies of a given subgraph:
a fixed subgraph $H$ is encoded by several different configurations.
Configuration in this sense cannot be encoded using homomorphisms due to the existence of an ordering. Hence, in this situation we do not have a V-property arising from our setting. Let us mention that the problem considered in \cite{FurHaj92} is a natural generalization of the Zarankievicz problem considered in \cite{KoSoTu54}: both problems coincide when the pattern is the all-ones grid, then any pattern in the matrix coincides with a copy of a complete bipartite graph in a given bipartite graph.

Secondly, in the context of groups, Solymosi in \cite[Theorem 2.2]{sol04} proved by means of the Triangle Removal Lemma of Ruzsa and Szemer\'edi \cite{RuzSze78} that for every $\delta > 0$ there is a threshold $n_0 \in \mathbb{N}$ such that if $G$ is a finite group of order $|G| \geq n_0$ then any set $B \subset  G \times G$ with $|B| = \delta |G|^2$ contains three elements $(a, b),\,(a,c),\,(e,f)$ such that $ab = ec$ and $ac = ef$. However, the corresponding V-property is not known.

Lastly, very recently a lot of effort has been devoted to transfer extremal results in the integers to the set of primes (see for instance \cite{ConFoZha15}). It would be very interesting to get similar results of the ones obtained here in this setting.

\appendix
\section{Example for uniformity}\label{s.app_uni}

The following is an example of a system for which the general arguments from Section~\ref{sec:random} exhibit a gap.

Consider
\begin{displaymath}
A=\begin{pmatrix}
1 & -2 & 1 & 0 & 0 & 0 & 0 &  0\\
1 &  1 & 1 & 1 & 1 & 1 & 1 & -7 \\
\end{pmatrix}	
\end{displaymath}
in $\Z_q$ large prime $q$. Now consider the following system of configurations
$\overline{S}=S\cup S'$
with $S=A^{-1}(0)\subset \Z_q^8$ and $S'$ having the following properties:
\begin{itemize}
	\item
$S'\subset \Z_q^8$.
\item $|S'|=q^c$ for $5>c> 4$ (take c=4.5).
\item $\pi_{\{4,5,6\}}(S')=\{(1,2,3)\}$.
\item For $U\subseteq [1,2,3]\cup[7,8]$, $|\pi_{U}(S')|=\Theta(q^{c\frac{|U|}{5}})$
and $\pi^{-1}_{U}(x)=\Theta(\max_{g\in \pi_{U}(S')}\{|\pi_{U}^{-1}(g)|\})$
\item all the elements in $S'$ have all the coordinates pairwise different.
\end{itemize}
Such a set $S'$ can be created by choosing elements in $\Z_q^5$ uniformly and random with probability $\frac{q^{c}}{q^5}$ and discarding the (few) unwanted elements: the resulting set well have the wanted probabilities with high probability.
In other words, $S'$ has the following shape: $S'=\{(\ast,\ast,\ast,1,2,3,\ast,\ast)\; : \; \ast\in \Z_q\}$, has size being $q^c$ and it is uniform distributed throughout the coordinates $[1,2,3]\cup[7,8]$.

We compute both the values $\alpha_i^{k}$ (for $i\in [1,8]$) and $\min_{\substack{|U|=i,}{U\subset [1,k]}}\left\{|\pi_{U}\left(S^{(k)}\cup S'\right)|\right\}=\chi_i^k$. The values different values can be found in the following table:

\begin{center}
\begin{tabular}{c||c|c|c|c|c|c|c|c|}

$i$ & 1 & 2  & 3 & 4 & 5 & 6 & 7 & 8 \\
  \hline
  $\alpha_i^k$ & $\theta(q^6)$ & $\theta(q^{4.5})$ & $\theta(q^{4.5})$ & $\theta(q^{3.6})$ & $\theta(q^{2.7})$ & $\theta(q^{1.8})$ & $\theta(q^{0.9})$ & $\theta(1)$ \\
\hline
$\chi_i^k$& &
	$\theta(q^2)$&
	$\theta(q^{2.7})$&
	$\theta(q^3)$&
	$\theta(q^4)$&
	$\theta(q^5)$&
	$\theta(q^6)$&
	$\theta(q^6)$\\
\end{tabular}
\end{center}

The first line of the table implies
	\begin{displaymath}
		\max_{\ell \in [2,8]} \left(\frac{\alpha_{\ell}^k}{\alpha_1^k}\right)^{\frac{1}{\ell-1}}=\theta\left( q^{-\frac{1}{4}}\right)=\left(\frac{\alpha_3^k}{\alpha_1^k}\right)^{1/2},
	\end{displaymath}
while the second has
	$$
		\max_{\substack{U\subseteq [1,k]\\ |U|\geq 2}}
		\left(\frac{|\Z_p|}{|\pi_U(S^{(k)}\cup S'|)}\right)^{\frac{1}{|U|-1}}=\theta\left( q^{-\frac{2}{3}}\right)= \left(\frac{q}{q^{3}}\right)^{\frac{1}{3}}
$$
as a consequence.
In particular, we see that there are systems for which Theorem~\ref{thm: random2} and Proposition~\ref{t.zero_statement} do not close the gap. In this example there are two different features involved. In one of them we see that $S'$ is such that one projection onto the coordinate variables $\{4,5,6\}$ has only one possible solution. Additionally, this solution has many preimages. This forces the Theorem~\ref{thm: random2} to pick a larger probability that what it should.

On the other hand, there are projections $\pi_U$ that have more values than what they should. This forces the probability in Proposition~\ref{t.zero_statement} to be smaller and compensate for this. Additionally, let us remark that, in general, the probability from Proposition~\ref{t.zero_statement} is a priori different than the one in Theorem~\ref{thm: random2}, hence our arguments force, in general, a gap between the 1-statement and the 0-statement.

\bibliography{RueSerraVena}
\bibliographystyle{abbrv}

\end{document}